\documentclass[absolute]{mymathart}
\usepackage{mymathmacros}
\usepackage{color,stmaryrd,caption,keyval,subfig}
\usepackage[shortlabels]{enumitem}
\usepackage{microtype}

\numberwithin{equation}{section}

\theoremstyle{changebreak}

\newcommand{\mc}[1]{\mathcal{#1}}

\newtheorem{thm}{Theorem}[section]%
\newtheorem{lem}[thm]{Lemma}%
\newtheorem{cor}[thm]{Corollary}%
\newtheorem{prop}[thm]{Proposition}%

\newcommand{\HH}{\mathbb{H}}
\newcommand{\DD}{\mathbb{D}}

\theoremstyle{defnbreak}
\newtheorem{rmk}[thm]{Remark}%
\newtheorem{defn}[thm]{Definition}%

\newcommand{\real}{\operatorname{real}}

\newcommand{\U}{\mathcal{U}}

\usepackage[hypertexnames=false]{hyperref}

\title{A bouquet of pseudo-arcs}
\author{Tania Gricel Benitez}
\address{Dept.\ of Mathematical Sciences, University of Liverpool, Liverpool L69 7ZL, UK} \email{t.benitez-lopez@liverpool.ac.uk}
\author{Lasse Rempe}
\address{Dept.\ of Mathematical Sciences, University of Liverpool, Liverpool L69 7ZL, UK, ORCiD: 0000-0001-8032-8580}
\email{l.rempe@liverpool.ac.uk}

\newcommand{\F}{\mathcal{F}}

\newcommand{\Blog}{\mathcal{B}_{\log}}

\newcommand{\s}{\underline{s}}

\begin{document} 

\begin{abstract}
 We prove the existence of a transcendental entire function
   whose Julia set is a ``bouquet of pseudo-arcs''. More precisely,
   $J(f)\cup\{\infty\}$ is an uncountable union of pseudo-arcs, which are
   pairwise disjoint except at infinity. 
   
  The existence of such a function follows from a more
   general result of the second author, but 
   our construction is considerably simpler and more explicit.
   In particular, the
   function we construct can be chosen to have lower order $1/2$, while the lower 
   order in the previously known example 
   is infinite. 
   \end{abstract}

\maketitle

\section{Introduction}\label{sec:intro}
  
  We consider the iteration of transcendental entire functions; i.e.\ of non-polynomial 
   holomorphic self-maps of the complex plane. 
   The subject was founded by Fatou in 1926 \cite{fatou}. In particular, he observed
   that the Julia sets of certain transcendental entire functions 
   contained curves of \emph{escaping points}, i.e., points whose orbits tend to
   infinity. Fatou asked whether this
   holds more generally. Eremenko made this question more precise in 1989, when 
   he asked whether every escaping point of a transcendental entire function 
   can be connected to infinity by a curve of escaping points. 
   The latter question was answered in the negative in~\cite{strahlen}. The counterexample
   belongs to a class of entire functions whose dynamics is of a particularly simple form. 
  
  \begin{defn}\label{defn:disjointtype}
    An entire function $f$ is of \emph{disjoint type} if there is a bounded Jordan domain $D\subset\C$ such
      that $f(\overline{D})\subset D$, and such that 
        \[ f\colon \mathcal{V} \defeq f^{-1}(\C\setminus D)\to \C\setminus D \]
      is a covering map. 
  \end{defn}  
  
  Equivalently, $f$ is \emph{hyperbolic} with connected Fatou set; see~\cite[Definition~1.1]{arclike} and~\cite[Lemma~3.1]{baranskikarpinskatrees}. 
    We refer 
    to~\cite{arclike} for background on disjoint-type entire functions and their significance for wider
    classes of transcendental entire functions. If $f$ is a disjoint-type entire function, then its
    Fatou set consists of a single immediate attracting basin. In particular, the \emph{escaping set}
    \[ I(f) \defeq \{z\in\C\colon f^n(z)\to\infty\} \]
    is contained in the Julia set $J(f)$. Thus the following theorem from~\cite{strahlen} does indeed
    give a negative answer to Eremenko's question. 
  
  \begin{thm}[{\cite[Theorem~1.1]{strahlen}}]\label{thm:strahlen1}
    There is a disjoint-type entire function $f$ such that $J(f)$ contains no curve to infinity. 
  \end{thm} 
  
   In \cite[Theorem~8.4]{strahlen}, the authors also sketch a proof of the following stronger result.
   
   \begin{thm}\label{thm:strahlen2}
    There is a disjoint-type entire function $f$ such that $J(f)$ contains no arc. 
   \end{thm}
   
   Every connected component $C$ of the Julia set $J(f)$ of a disjoint-type entire function $f$ 
     is an unbounded closed connected set. Thus
     $\hat{C}\defeq C\cup\{\infty\}$ is a compact connected set, which, following~\cite{arclike},
     we call a \emph{Julia continuum}. 
     For the function in Theorem~\ref{thm:strahlen2}, each such continuum 
     may be considered ``pathological'' in that it contains
     no arc, and it is natural to ask how complicated the topology of a Julia continuum
     may become. In particular, we may ask whether such a continuum may be
     \emph{hereditarily indecomposable}, i.e., have the property that any two subcontinua
     are either nested or disjoint. Clearly a hereditarily indecomposable continuum contains no arcs. 
     
   A famous example of a hereditarily indecomposable continuum is the \emph{pseudo-arc} (see Section~\ref{sec:topology}). 
     So we may ask, in particular, whether the pseudo-arc may arise as the Julia continuum of a disjoint-type
     entire function.\footnote{In fact, it follows from~\cite[Theorem~1.4]{arclike} and~\cite{hoehnoversteegenhomogeneous} that
     any hereditarily indecomposable Julia continuum of a disjoint-type function is a pseudo-arc.}
     This question was answered by the second author in~\cite[Theorem~1.5]{arclike}.
     
     \begin{thm}\label{thm:pseudoarc1}
       There is a disjoint-type entire function $f$ such that every Julia continuum of $f$ is a pseudo-arc. 
     \end{thm} 
     
    In fact,~\cite[Theorem~2.7]{arclike} shows that any continuum that arises as the inverse limit of a self-map of the interval below the identity
      (see~\cite[Definition~2.6]{arclike}) can be realised as a forward-invariant Julia continuum. Theorem~\ref{thm:pseudoarc1} is
      then obtained as a corollary, using a classical theorem of Henderson~\cite{hendersonpseudoarc}
      that states that the pseudo-arc can be represented as such an inverse limit. 
      The goal of this article is to give a simpler and more 
      direct proof of Theorem~\ref{thm:pseudoarc1}, by directly applying some of the ideas from~\cite{hendersonpseudoarc}. 
      This may also be the simplest proof of Theorem~\ref{thm:strahlen2} yet. (Making the proof sketched in~\cite{strahlen} precise would 
      require a considerable amount of book-keeping and notation.) 
      Since our approach is much more 
      explicit, it offers the possibility of obtaining more detailed information about the function-theoretic behaviour of such a ``pathological'' function $f$. 
      In particular, we show the following.
      
   \begin{thm}\label{thm:pseudoarc2}
    The function $f$ in Theorem~\ref{thm:pseudoarc1} can be chosen to have lower order of growth $1/2$; that is,
       \[ \liminf_{r\to\infty} \frac{\log \log M(r,f)}{\log r} = \frac{1}{2}, \]
       where $M(r,f)= \max\{\lvert f(z)\rvert\colon \lvert z\rvert = r\}$. 
   \end{thm}
    In contrast, the construction in~\cite{arclike} leads to functions of infinite lower order. 
      Using the techniques of \cite[Section~18]{bishopfolding} or \cite[Section~15]{arclike},
       the function in Theorem~\ref{thm:pseudoarc2} can also be constructed to have only two critical 
      values and no finite asymptotic values. 
      
  \begin{thm}[Pseudo-arc bouquets in the Speiser class]\label{thm:speiser}
   The function $f$ from Theorems~\ref{thm:pseudoarc1} and~\ref{thm:pseudoarc2} can be chosen such that 
     $f$ has exactly two critical values and no finite asymptotic values, and such that there is $D>1$ such that
     all critical points
     are of degree at most $D$. 
  \end{thm}
  
   Let us also note a consequence of Theorem~\ref{thm:speiser}. In polynomial
    dynamics, \emph{local connectivity} of Julia sets is an important and much-studied property.
    The reason for this is that, when $J(f)$ is locally connected, 
    the topological dynamics of the polynomial $f$ can be completely described
    in combinatorial terms; compare e.g.\ \cite{pincheddisk}. Local connectivity 
    of Julia sets has also been studied for transcendental entire functions. Of course, when
    $J(f)=\C$, the Julia set is locally connected, but this does not imply that the dynamics
    is simple! The following example shows that, even when $J(f)\neq \C$, local connectivity of the
    Julia set does not imply simple dynamics.
    
  \begin{cor}[Locally connected Julia sets with pseudo-arcs]\label{cor:locallyconnected}
    There exists a transcendental entire function $f$ with the following property. 
    \begin{enumerate}[(a)]
      \item $J(f)$ is a Sierpi\'nski carpet, and in particular locally connected.
      \item $J(f)$ contains an infinite collection $P_1, P_2,\dots$ of pairwise disjoint unbounded connected invariant sets such that
                         $P_j\cup\{\infty\}$ is a pseudo-arc for all $j$. 
      \item $J(f)$ contains an uncountable collection $\mathcal{P}$ of pairwise disjoint unbounded connected sets
                such that, for all $P\in\mathcal{P}$, $P\cup\{\infty\}$ is a pseudo-arc and $f(P)\in \mathcal{P}$. 
    \end{enumerate}
  \end{cor}
   \begin{remark}[Remark~1]
     This phenomenon is also remarked upon (without proof) 
       in \cite[Paragraph~after~Corollary~1.9]{hyperbolicboundedfatou} 
      and \cite[Discussion~after~Theorem~2.11]{arclike}. 
   \end{remark}
   \begin{remark}[Remark~2]
     Similar phenomena are known for rational maps, where it is possible e.g.\ that the Julia set
       is a Sierpi\'nski carpet, while the map may have a Cremer fixed point; see~\cite{roeschnewton}.
       In contrast, our example is \emph{hyperbolic} (see Section~\ref{sec:speiser}). 
   \end{remark}

 Let us outline the basic idea of the various constructions mentioned above, as well
    as ours, and comment on the differences between them. 
    The connected components of the set $\mathcal{V}$ as in Definition~\ref{defn:disjointtype} are called
     the \emph{tracts} of $f$; each such tract $V$ is simply connected and mapped by $f$ as a universal covering. 
     The proofs of the above-mentioend theorems from~\cite{strahlen} and~\cite{arclike}~-- and also our proof of 
     Theorems~\ref{thm:pseudoarc1} and~\ref{thm:pseudoarc2}~-- proceed by first constructing a suitable 
     simply-connected domain $T$ with 
        \begin{equation}\label{eqn:halfplane}
           \overline{T}\subset \HH \defeq \{x + iy\colon a > 0 \} \end{equation}
             and a conformal isomorphism 
        \[ F \colon T\to \HH, \]
     where $T$ is disjoint from its $2\pi i\Z$-translates. The map $F$ is constructed so that the universal covering map
      \[ f \colon \exp(T)\to  \exp(\HH) = \{z\in\C\colon \lvert z\rvert > 1\}; \qquad f(\exp(\zeta)) = \exp(F(\zeta)) \] 
        and its ``Julia set'' 
          \[ J(f) \defeq \{ z\in\C\colon f^n(z)\text{ is defined and of modulus $\geq 1$ for all $n\geq 0$}\} \]
           have the desired properties. 
    Then an approximation result (see e.g.\ Theorem~\ref{thm:realisation} below)  is applied in order to obtain
    a disjoint-type entire function $g$ for which $J(g)$ is homeomorphic to $J(f)$. 

 The tract $T$ used for the proof of Theorem~\ref{thm:strahlen1} consists of a long straight half-strip, into which a countable
    number of 
   ``wiggles'' are inserted. Here by a wiggle, we mean a winding strip that first increases to a large real part $R$, decreases back down to a (much) smaller real part $r$,
     and finally begins to grow again. See Figure~\ref{fig:strahlen1}. By thickening the intermediate straight pieces, the corresponding
     function can be made to have lower order $1/2$; see \cite[Proposition~8.3]{strahlen} and~\cite[Theorem~1.10]{approximationhypdim}. 
     The construction in the proof of Theorem~\ref{thm:strahlen2}, which is
      only sketched in~\cite{strahlen}, is considerably more complicated.
      Here, the ``wiggles'' are each made up of further ``subwiggles'' 
       (see Figure~\ref{fig:strahlen2}). This construction is iterated to
       a greater and greater depth the further to the right of the tract
       these wiggles are inserted. 
       The proof of Theorem~\ref{thm:pseudoarc1} in~\cite{arclike} 
      uses an even more elaborate construction; see \cite[Figure~9]{arclike}. 

\begin{figure}
\begin{center}
 \includegraphics[width=\textwidth]{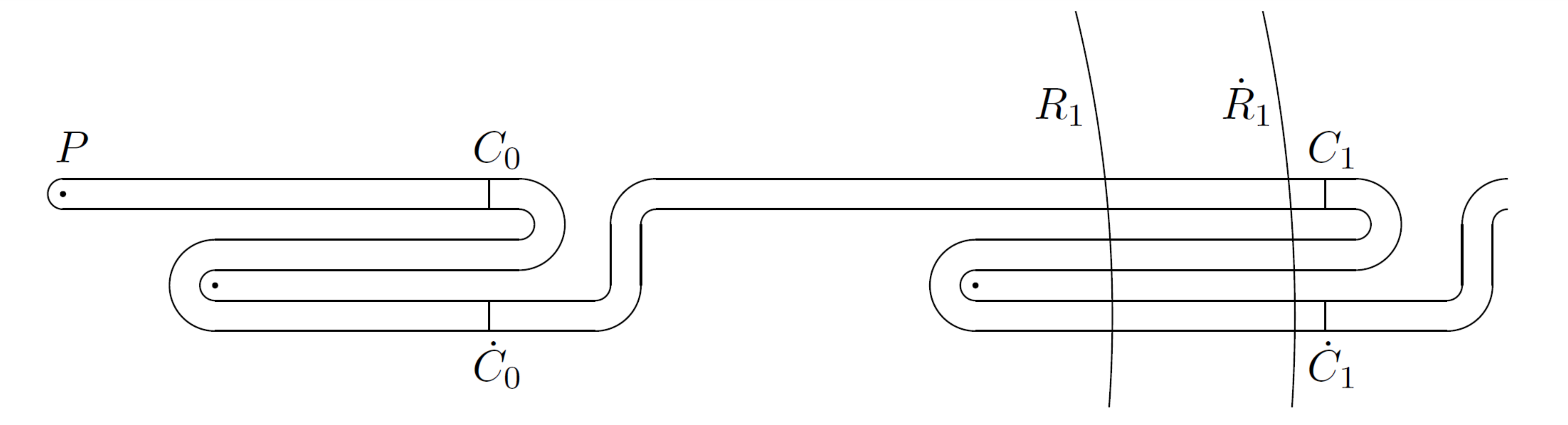}
\end{center}
\caption{\label{fig:strahlen1}The tract used in the proof of Theorem~\ref{thm:strahlen1}; reproduced from~\cite[Figure~5]{strahlen}.}
\end{figure}

\begin{figure}
\begin{center}
\includegraphics[width=\textwidth]{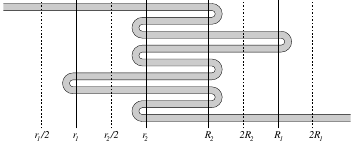}
\end{center}
\caption{\label{fig:strahlen2}The tract used in the proof of Theorem~\ref{thm:strahlen2}; reproduced from~\cite[Figure~8]{strahlen}.}
\end{figure}
      
   In contrast, our proof of Theorem~\ref{thm:pseudoarc1} in this paper can be made to work using exactly the same type of tract as in 
   Figure~\ref{fig:strahlen1} (naturally, with different choices of where to place the ``wiggles''). We use a slight 
   modification, also used in~\cite[Figures~42 and~43]{bishopfolding}, to 
%   realise the function with a finite set of singular values (Section~\ref{sec:speiser}) and 
   ensure lower order $1/2$; see Figure~\ref{fig:tracts}.
   
\subsection*{Structure of the paper} 
In Section~\ref{sec:topology}, we collect necessary background concerning arc-like continua and pseudo-arcs. We also give a brief overview of the history of the pseudo-arc. 
    Section~\ref{sec:models} introduces the class of conformal isomorphisms that are used as models for our construction, and discusses their fundamental properties.
    In Section~\ref{sec:projection1}, we introduce the \emph{one-dimensional projection}
    $\phi$
       of such a function $F$, a crucial tool for our proof that encodes the essential mapping 
    behaviour of the inverse $F^{-1}$ via
    a (usually non-injective) function of one real variable. Section~\ref{sec:condition} gives 
     a sufficient condition on the one-dimensional projection 
    to ensure that all Julia continua are pseudo-arcs. The remainder of the paper is 
    dedicated to constructing a conformal isomorphism with these properties. In
    Section~\ref{sec:projection}, we consider the mapping properties of the projection
    $\phi$ further. 
    
    Section~\ref{sec:tracts} sets up our main construction,
    which is inductive. The idea is that we start with a straight half-strip, and inductively continue inserting additional wiggles into this tract, which cause more and more
    ``crooked'' mapping behaviour of the one-dimensional projection. In this way, we obtain a sequence of simply connected domains $T_n$, and associated functions
    $F_n\colon T_n\to\HH$. A key fact in the construction is that the map $F_{n+1}$ will be close to $F_n$, as long as the new wiggle is inserted far enough to the right.
    Finally, the proofs of our main theorems are carried out in Section~\ref{sec:pseudoarcs}.

\subsection*{Notation.} As usual, $\C$ and $\Ch=\C\cup\{\infty\}$ 
  denote the complex plane and the Riemann sphere, respectively. 
  Recall from~\eqref{eqn:halfplane} that $\HH$ is the right half-plane. 
  Euclidean distance, diameter and length are denoted by $\dist$, $\diam$ and $\ell$, 
  respectively. The open Eulidean disc with centre $z_0$ and radius $r$ is denoted
    \[ D(z_0,r) \defeq \{z\in\C\colon \lvert z - z_0\rvert < r. \]
  The closure in $\C$ of a set $X\subset \C$ is denoted $\overline{X}$;
    the closure of $D(z_0,r)$ is denoted $\overline{D}(z_0,r)$. 

\section{Background from continuum theory}\label{sec:topology}

  A \emph{continuum} $X$ is a non-empty compact connected metric space. The continuum $X$ 
    is \emph{non-degenerate} if it contains more than one point. We refer to~\cite{continuumtheory} for an introduction to continuum theory. 
  
  \begin{defn}[Arc-like continua]
    A continuum $X$ is \emph{arc-like} if, for every $\eps>0$,  
     there exists a surjective continuous map $g\colon X \to [0,1]$ 
      such that 
        \[ \max_{t\in [0,1]} \diam_X(g^{-1}(t)) \leq \eps. \]
        Such $g$ is called an \emph{$\eps$-map}. 
  \end{defn}
  \begin{remark}
   Here $\diam_X$ denotes the diameter 
    with respect to the given metric on $X$. 
    Since $X$ is compact, equivalent metrics on $X$ give rise to the 
    same notion of arc-likeness, so that being arc-like is a topological property. All continua 
    considered in this paper are subsets of the Riemann sphere,
    and are endowed with the usual spherical metric. 
  \end{remark}

The second notion from continuum theory that we use in this paper 
  is that of \emph{hereditarily indecomposable continua}.   
  \begin{defn}[Hereditarily indecomposable continua]
    A continuum $X$ is \emph{hereditarily indecomposable} if,
      for any two subcontinua $X_1,X_2\subset X$ with $X_1\cap X_2 \neq \emptyset$,
     either 
      $X_1\subset X_2$ or $X_2\subset X_1$. 
  \end{defn}
 
   Knaster~\cite{knasterpseudo}, in his thesis, was the first to construct
     an example of a hereditarily indecomposable continuum. Moise~\cite{moisepseudo} answered a long-standing 
     open question of Mazurkiewicz by constructing
     a hereditarily indecomposable 
     continuum that is homeomorphic to each of its non-degenerate subcontinua. 
     Moise called
     this continuum a \emph{pseudo-arc}. Building on Moise's construction, Bing~\cite{bingpseudo}
     showed that Moise's pseudo-arc $X$ is \emph{homogeneous}; i.e.\ for any pair of points in $X$ there is a homeomorphism 
     of $X$ to itself that maps one to the other. This answered a long-standing question of Knaster and Kuratowski. 

  All these examples are arc-like by construction; Moise himself remarked that his
   construction was very similar to Knaster's and asked
    whether the two continua are homeomorphic. Bing~\cite{bingpseudocharacterisation}
     gave a positive answer by 
     showing that any two arc-like hereditarily indecomposable continua
     are homeomorphic. Thus we use the following definition. 
     
   \begin{defn}[Pseudo-arcs]
      A \emph{pseudo-arc} is a hereditarily indecomposable arc-like continuum. 
   \end{defn}
   
   We refer to~\cite{lewispseudoarc} for further background on the pseudo-arc. 

\section{A family of disjoint-type models}\label{sec:models}

 As in~\cite{strahlen} and~\cite{arclike}, we construct our examples by approximating
   suitable models,
   constructed in logarithmic coordinates. More precisely, these models have the following general form. 
   (See Section~\ref{sec:tracts} for the specific functions used in our construction.) 
   
   \begin{defn}[Class $\mathcal{H}$]\label{defn:classH}
     We denote by $\mathcal{T}$ the set of all simply-connected domains
        \begin{equation}\label{eqn:strip}
           T\subset S \defeq \{x + iy \colon x > 4, \lvert y \rvert < \pi \} \end{equation}
      with $5\in T$, where
      \begin{enumerate}[(a)]
        \item $T$ is unbounded, and $\partial T$ is locally connected; 
        \item there is only one access to infinity in $T$; i.e., any two curves connecting the same finite endpoint
           to infinity in $T$ are homotopic.\label{item:access}
      \end{enumerate}

     If $T\in\mathcal{T}$, then there is a unique conformal isomorphism
        $F\colon T\to \HH$ with $F(5)=5$ that extends continuously to $\infty$ with $F(\infty)=\infty$. We denote the class
        of all such functions $F$ by $\mathcal{H}$, and 
        the domain of $F\in\mathcal{H}$ by $T(F)$.            

     For $\nu>0$, we also consider the
        subclass $\mathcal{H}_{\nu}$ of $F\in\mathcal{H}$ such that 
           \[ \diam(F^{-1}(\{w\in\HH\colon \lvert w \rvert = R \})) \leq \nu \] 
           for all $R>0$. In other words, the geodesics of $T$ that are 
           perpendicular to $F^{-1}\bigl((0,\infty)\bigr)$ have uniformly bounded
           Euclidean diameter. We will call these the 
          \emph{``vertical'' geodesics} of $T$ associated to $F$. 
   \end{defn}
   \begin{rmk}
     Let $T\in\mathcal{T}$; we justify the existence of the conformal isomorphism $F$ in the definition. 
     There is a conformal isomorphism $F\colon T\to\HH$ satisfying $F(5)=5$, and this map is unique up to 
     postcomposition by a M\"obius transformation of $\HH$ fixing $5$. 
     By the Carath\'eodory-Torhorst theorem, \cite[Theorem~1.7]{pommerenke}, 
     $F^{-1}\colon \HH\to T$ extends
     continuously to
     $\overline{\HH}\cup\infty$. Property~\ref{item:access} 
     in the definition of $\mathcal{T}$ 
     implies that there exists exactly one point $\zeta \in \partial{\HH}\cup\infty$ 
     such that
     $F^{-1}(\zeta)=\infty$; by postcomposing $F$ with a M\"obius transformation, 
     we may assume that $\zeta=\infty$, which makes
     $F$ unique. Since $\overline{\HH}\cup\{\infty\}$ is compact, it follows that $F$ 
     itself extends continuously to infinity.
      (However, unless $T$ is a Jordan domain, $F$ does not extend continuously to 
      all finite boundary points.)
    \end{rmk}
    
 In order to estimate functions $F\in \mc{H}$, we frequently use that conformal isomorphisms are isometries of the corresponding
  hyperbolic metrics; compare~\cite{beardonminda} for background on plane hyperbolic geometry. 
  The density $\rho_{\HH}$ of the hyperbolic metric on the right half-plane is given by
     \[ \rho_{\HH}(\zeta) = \frac{1}{\re \zeta}. \]
  Furthermore, by monotonicity of the hyperbolic metric, 
    the density $\rho_T$ of any tract $T\in\mathcal{T}$ is bounded below by the density
    $\rho_{\tilde{S}}$ of the bi-infinite strip \[ \tilde{S}\defeq \{ a + ib\colon \lvert b\rvert < \pi \}\supset S\supset T. \] So 
     \begin{equation}\label{eqn:stripestimate} \rho_T(z) \geq \rho_{\tilde{S}}(z) = 1/(2\cos(\im z /2)) \geq 1/2 \end{equation}
   for all $z\in T$. (In fact,~$\rho_T(z) \geq 1/2$ holds for any simply-connected domain $T$ that is disjoint from its
    $2\pi i\Z$-translates, even without the assumption that $T\subset\tilde{S}$. Compare~\cite[Corollary~2.2]{ELestimate}.) In particular, functions in $\mc{H}$ uniformly expand the Euclidean metric. 
  
  \begin{lem}[Expansion for $F\in\mc{H}$]\label{lem:expansion}
    Suppose that $F\in\mathcal{H}$. Then $\lvert F'(z)\rvert \geq \re F(z)/2$ for all $z\in T$. In particular, 
       \[ \lvert F'(z)\rvert \geq 2 \]
     for all $z\in T$ with $\re F(z) \geq 4$.
  \end{lem}   
  \begin{proof} (See also~\cite[Lemma~2.1]{ELestimate}.)
    $F$ is an isometry between $T$ and $\HH$ with their respective
      hyperbolic metrics, so 
         \[ \lvert F'(z)\rvert = \frac{\rho_T(z)}{\rho_{\HH}(F(z))} \geq \frac{\re F(z)}{2}.\qedhere \]
  \end{proof}

  Given a function $F\in\mathcal{H}$, we may extend $F$ to a $2\pi i$-periodic function
    \[ \hat{F}\colon T + 2\pi i\Z\to \HH. \]
     Note that $\hat{F}^{-1}$ has countably many branches, defined by
      \[ \hat{F}_s^{-1}(z) \defeq F^{-1}(z) + 2\pi i s. \] We are interested in
     the set of points that remain in the half-plane under iteration of $\hat{F}$, 
     and its decomposition
      into individual components according to symbolic dynamics. 
  
 \begin{defn}[Julia continua of $\hat{F}$]\label{defn:juliacontinua}
   Let $F\in\mathcal{H}$ and let $\hat{F}$ be the $2\pi i$-periodic
     extension of $F$. We define
       \[ J(\hat{F}) \defeq \{z\in \HH\colon \hat{F}^n(z) \text{ is defined and belongs to $T+2\pi i \Z$ for all $n$} \}. \]
       
   If $\s = s_0 s_1 s_2\dots$ is a sequence of integers, we define
      \[ J_{\s}(\hat{F}) \defeq \{z\in J(\hat{F})\colon \hat{F}^n(z) \in T + 2\pi i s_n \text{ for all $n\geq 0$}\}. \] 
     If $J_{\s}(\hat{F})\neq \emptyset$, we call $\hat{J}_{\s}(\hat{F}) \defeq J_{\s}(\hat{F})\cup\{\infty\}$ a \emph{Julia continuum} of $\hat{F}$.
 \end{defn}
 
 The following is a special case of~\cite[Proposition~7.6]{arclike}.
 
 \begin{lem}[Julia continua are arc-like]
   Every Julia continuum of $\hat{F}$ is an arc-like continuum.
 \end{lem} 
 \begin{proof}
  We can write $J_{\s}(\hat{F})$ as a nested intersection of compact, connected sets as follows.
    For $n\geq j\geq 0$, define inductively 
     $X_n^n \defeq (\overline{S}+2\pi i s_n)\cup\{\infty\}$ and 
    \[ X_n^{j}\defeq \hat{F}_{s_n}^{-1}(X_{n+1}^{j}). \]
    (Recall from~\eqref{eqn:strip} that $S$ is a half-strip containing $T$.)
   Each $X_n^j$ is a continuum, as the image of a continuum under a continuous function,
   and $X_n^{j+1}\subset X_n^j$. So 
       \[ J_{\s}(\hat{F}) = \bigcap_{j=0}^{\infty} X_0^j \]
    is  a nested intersection of continua, and thus also a continuum.
    
   To see that $J_{\s}(\hat{F})$ is arc-like, define
      \[ g_n\colon \hat{J}_{\s}(\hat{F}) \to [4,\infty]; \qquad z \mapsto \begin{cases} \re \hat{F}^{n}(z)   &\text{if }z\in\C  \\      
                                                                                \infty  &\text{if } z = \infty.\end{cases} \]
   Then $g_n$ is a continuous function onto a non-degenerate compact interval. If $t$ is a finite point in the range of $g_n$, and
   $\zeta\in J_{\s}(\hat{F})$ with $g_n(\zeta)=t$, then 
      \[ \diam g_n^{-1}(t)  \leq \diam \hat{F}^{-n}( \{\re \zeta + it\colon (2s_n-1)\pi \leq t \leq (2s_n+1)\pi \} ) \leq 2^{-n}\cdot 2\pi. \]
      So the Euclidean diameter, and hence also the spherical diameter, of $g_n^{-1}(t)$
       tends to zero uniformly in $t$ 
      as $n\to\infty$, while $g_n^{-1}(\infty)$ consists of a single point.
      Therefore the continuum is arc-like. 
 \end{proof}
  
 By expansion of $F$, and since $T$ is contained in a strip of bounded height, the real parts of points in the same Julia continuum eventually separate under iteration. 
  \begin{lem}[Separation of orbits]\label{lem:separation}
  Let $F\in\mathcal{H}$, and let $z,w \in J(\hat{F})$ 
    belong to the same Julia continuum. Then 
     \[ \lvert \hat{F}^n(z) - \hat{F}^n(w)\rvert \geq 2^n\cdot \lvert z-w\rvert. \]
   In particular, $\lvert \re \hat{F}^n(z) - \re \hat{F}^n(w)\rvert \to \infty$. 
 \end{lem} 
 \begin{proof}
   It is enough to prove the result for $n=1$; the general case then follows by induction. 
     Connect $z_1 \defeq \hat{F}(z)$ and $w_1 \defeq \hat{F}(w)$ 
     by a straight line segment $\gamma$.
     Then $(F^{-1})'(\zeta) \leq 1/2$ for all $\zeta\in\gamma$ by 
     Lemma~\ref{lem:expansion}. Since $z$ and $w$ are in the same Julia continuum,
     we have $z - w = F^{-1}(z_1) - F^{-1}(w_1)$. Thus 
       \[ \lvert z - w \rvert = \lvert F^{-1}(z_1) - F^{-1}(w_1)\rvert \leq
           \ell(F^{-1}(\gamma)) \leq \frac{1}{2} \lvert z_1 - w_1 \rvert.\qedhere \]
 \end{proof}  
 
 We shall also use the following fact. If $F\in\mathcal{H}_{\nu}$, and we fix finitely many points in the same Julia continuum, then under a sufficiently large iterate of $\hat{F}$, all but one 
  will eventually lie in a sector around the real axis. More precisely: 
 \begin{lem}[Orbits enter a sector]\label{lem:sector}
    For every $\nu>0$, there is $\delta>0$ with the following property.
      Suppose that $F\in \mathcal{H}_{\nu}$, and
       $z,w\in T(F)$, with $\lvert z - w\rvert \geq \delta$,
     $\lvert F(z) \rvert \geq \lvert F(w)\rvert\geq 4$, and $\lvert \im F(z) - \im F(w)\rvert \leq 2\pi$. Then 
     $\re F(z) \geq \lvert \im F(z)\rvert + 2\pi$. 
 \end{lem}
 \begin{proof}
   Set $\delta\defeq 2\cdot (\nu + \log(2 + 3\pi/2))$ 
     and let $z$ and $w$ be as in the statement of the lemma. 
      Set $\tilde{z} \defeq F^{-1}(\lvert F(z)\rvert)$ and $\tilde{w}\defeq F^{-1}(\lvert F(w)\rvert )$. 
     Since $F\in\mathcal{H}_{\nu}$, we have 
      $\lvert\tilde{z}-z\rvert \leq \nu$, and likewise $\lvert \tilde{w}-w\rvert \leq \nu$. 
     In particular,
        \[ \lvert \tilde{z} - \tilde{w} \rvert \geq \delta - 2\nu, \] 
     and hence the hyperbolic distance in $T$ between $\tilde{z}$ and $\tilde{w}$ satisfies
        \[ \tilde{\delta} \defeq \dist_T(\tilde{z},\tilde{w}) \geq \frac{\delta-2\nu}{2} =
             \log\left( 2 + \frac{3\pi}{2}\right), \]
        by~\eqref{eqn:stripestimate} and choice of $\delta$. 
        
   Since $F$ is a conformal isomorphism, the hyperbolic distance in $\HH$
     between $\lvert F(z)\rvert$ and $\lvert F(w)\rvert$ is
    also $\tilde{\delta}$, so 
       \[ \lvert F(z)\rvert = \exp(\tilde{\delta}) \cdot \lvert F(w)\rvert
            \geq \left(2 + \frac{3\pi}{2}\right)\cdot \lvert F(w)\rvert. \] 
    Thus
        \begin{align*} \re F(z) &\geq  \lvert F(z)\rvert - \lvert \im F(z)\rvert \geq 
               (2 + 3\pi/2)\cdot \lvert F(w)\rvert - \lvert \im F(z)\rvert \\ &\geq 
                2\cdot\lvert \im F(w)\rvert   + 6\pi - \lvert \im F(z)\rvert \geq 
                 \lvert \im F(z)\rvert + 2\pi. \qedhere \end{align*}
 \end{proof} 

 To recover our results in the form that they are stated in the introduction, we 
   must pass from a conformal isomorphism 
   $F\in\mathcal{H}$ to an entire function $f$ with the same
   mapping properties. This is made possible by the following theorem,
   which is a combination of results of Bishop~\cite{bishopclassBmodels} and the
   second author~\cite{boettcher}, as explained in~\cite[Theorem~3.5]{arclike}.
   
 \begin{thm}[Realisation of models]\label{thm:realisation}
   Let $F\in\mathcal{H}$. Then there is a disjoint-type entire function $f$ such that every Julia continuum 
     of $f$ is homeomorphic to a Julia continuum of $F$, and vice versa. 
 \end{thm}  
 \begin{proof}
   Let $\hat{F}$ be again the $2\pi i$-periodic extension of $F$. Let $\HH_{1}$ denote the right half-plane
     $\{a+ib\colon a> 1\}$. The restriction of 
     $\hat{F}$ to $\hat{F}^{-1}(\HH_{1})$ is a disjoint-type function in the class
     $\Blog^{\operatorname{p}}$ as defined in~\cite[Definition~3.3]{arclike}. The claim 
     follows from~\cite[Theorem~3.5]{arclike}. 
 \end{proof}

 \section{The one-dimensional projection}\label{sec:projection1}

Lemma~\ref{lem:sector} means that, for $F\in \mathcal{H}_\nu$, with some $\nu>0$, the mapping properties of $F$ are essentially related to the behaviour of $F^{-1}$ on the
  real axis. (See Lemma~\ref{lem:Fandphi}.) This motivates the following definition. 

\begin{defn}[One-dimensional projection]\label{defn:phi}
 Given $F\in\mathcal{H}$, we call 
   \begin{equation}\label{eqn:phidefn} \phi \colon [4,\infty] \to [4,\infty], \quad \phi(t) \defeq \re F^{-1}(t) \end{equation} 
   the \emph{one-dimensional projection} of $F$.
\end{defn}

\begin{lem}[Properties of $\phi$]\label{lem:phiproperties}
  For every $F\in\mathcal{H}$, the one-dimensional projection $\phi$ has the following properties.
  \begin{enumerate}[(a)]
    \item $\displaystyle{\lvert \phi(t_1) - \phi(t_2)\rvert \leq \frac{\lvert t_1 - t_2\rvert}{2}}$ when $t_1,t_2\geq 4$.\label{item:phicontraction}
    \item In particular, $\phi(t)<t$ for $t\geq 5$, and $\phi(t) < 6$ for $t\in[4,5]$.
    \label{item:phidecreasing}
    \item $\displaystyle{\phi(t) < 5 + 2(\log t - \log 5)}$ for all $t\geq 5$.\label{item:phiestimate}
    \item In particular $\phi(t) < t-1$ for $t\geq 7$ and $\phi(t) < t/2$ for $t\geq 15$.\label{item:simplephiestimates}
   \end{enumerate}
\end{lem}
\begin{proof}
  If $t_1,t_2\geq 4$, then by Lemma~\ref{lem:expansion}, we have $\lvert (F^{-1})'(x)\rvert \leq 1/2$ for $x\in [t_1,t_2]$. Hence 
     \[ \lvert \phi(t_1) - \phi(t_2)\rvert = \lvert \re (F^{-1}(t_1) - F^{-1}(t_2))\rvert \leq 
              \lvert F^{-1}(t_1) - F^{-1}(t_2)\rvert \leq \frac{\lvert t_1 - t_2\rvert}{2}. \]
    This proves~\ref{item:phicontraction}. 
    In particular, if $t\geq 5$, then $\phi(t)-5\leq (t-5)/2$, and therefore
    $\phi(t)<t$. Similarly, if $\lvert t-5\rvert \leq 1$, then
       $\lvert \phi(t) - 5\rvert \leq 1/2$. This proves~\ref{item:phidecreasing}.  
              
   Now let $t\geq 5$. By~\eqref{eqn:stripestimate}, and since $F(5)=5$, we have
      \[ \lvert F^{-1}(t) - 5\rvert \leq 2\dist_T(F^{-1}(t),5) = 2\dist_{\HH}(t,5) = 2(
         \log t - \log 5). \]
    Hence 
     \[ \phi(t) \leq 5 + \lvert F^{-1}(t) - 5\rvert \leq 5 + 2(\log t - \log 5). \]
     This proves~\ref{item:phiestimate}, and it is
      easy to check numerically that~\ref{item:simplephiestimates} follows. 
f\end{proof}

 We also need the following fact about the relation between $F^{-1}(z)$ and $\phi(\re z)$, for points $z$ that are not necessarily on the real axis. 
  \begin{lem}\label{lem:Fandphi}
   Let $z\in\HH$ with $\re z \geq 4$ and $\lvert \im z \rvert \leq \re z + 2\pi$. 
    Then 
      \[ \lvert \re F^{-1}(z) - \phi(\re z)\rvert \leq 6. \]
  \end{lem}
  \begin{proof}
     The assumption implies that 
       \[ \dist_{\HH}(z,\re z) \leq \frac{\lvert \im z \rvert}{\lvert \re z\rvert} \leq 1 + \frac{2\pi}{\re z}
         \leq 1 + \frac{\pi}{2} < 3. \]
          By~\eqref{eqn:stripestimate}, 
      \begin{align*}
      \lvert \re F^{-1}(z) - \phi(\re z)\rvert &\leq 
       \lvert F^{-1}(z) - F^{-1}(\re z)\rvert \\ &\leq 
       2\dist_T(F^{-1}(z), F^{-1}(\re z)) 
        = 2 \dist_{\HH}(z,\re z) < 6. \qedhere \end{align*}
  \end{proof}

\section{A sufficient condition for pseudo-arc continua}\label{sec:condition}

 \begin{defn}[Quadruples]
   By a \emph{quadruple}, we always mean a tuple of four different real numbers $\geq 9$. We denote such 
     a quadruple in increasing order as $Q = (A<B<C<D)$. 
     We also define the \emph{size} of the quadruple as
       \[ \lvert Q\rvert = \min( A-5,B-A,C-B,D-C). \]
 \end{defn}

Following~\cite{hendersonpseudoarc}, we use the following terminology. 

\begin{defn}[Crooked iterates of $\phi$]
We say that $\phi^k$ maps 
  a  closed interval $I$ \emph{crookedly} over 
   the quadruple $(A<B<C<D)$, if 
       $\phi^k(I)\supset [A,D]$, and furthermore the convex hull of $I\cap \phi^{-k}(B)$ intersects the
       convex hull of $I\cap \phi^{-k}(C)$. 
 \end{defn}
\begin{remark}
  The final condition means that two points of $I\cap \phi^{-k}(B)$ surround a point of $I\cap \phi^{-k}(C)$, or vice versa. 
\end{remark}

The following result gives a sufficient condition for all Julia continua of a function $F\in\mathcal{H}_{\nu}$ to be a pseudo-arc. 
 Compare~\cite[Lemma~1]{hendersonpseudoarc}. 

\begin{thm}[Sufficient condition for pseudo-arcs Julia continua]\label{thm:pseudoarcrealisation}
  Let $F\in\mathcal{H}_{\nu}$, for some $\nu>0$. Suppose that
   there exists a constant $K>0$ such that the following property holds for all integer quadruples
     $Q = (A<B<C<D)$ with $\lvert Q\rvert \geq K$. 
     
   There exists $k_0=k_0(Q)\in \N$ such that, for every $k\geq k_0$ and
    every compact interval $I\subset [6,\infty)$ with
     $\phi^{k}(I)\supset [A,D]$, 
     the iterate     $\phi^k$ maps $I$ crookedly over $Q$. 
     
     Then every Julia continuum of the $2\pi i$-periodic extension $\hat{F}$ is a pseudo-arc. 
\end{thm} 
\begin{remark}
 By \cite[Corollary~8.7]{arclike}, it would be enough to prove the result for the invariant set of the conformal isomorphism $F$, i.e.\ for the
  Julia continuum of $\hat{F}$ 
  at address $\s = 0 0 0 \dots$. To remain self-contained, we instead give an argument for
  arbitrary Julia continua, avoiding the results of \cite[Section~8]{arclike}. 
\end{remark}
\begin{proof}
  Suppose, by contradiction, that $\hat{J}_{\s}(\hat{F})$ is a Julia continuum, and 
   that there are subcontinua $\mathcal{Z}^0,\mathcal{Z}^1 \subset \hat{J}_{\s}(\hat{F})$ 
   whose intersection is non-empty, but such that neither is contained in the other.
   Set $\mathcal{C}\defeq \mathcal{Z}^0\cup \mathcal{Z}^1$;
     $\mathcal{C}$ is a continuum by assumption. 
    We define
      \[ \mathcal{Z}^j_n \defeq \hat{F}^n(\mathcal{Z}^j) \subset T(F) + 2\pi i s_n \]
      and $\mc{C}_n\defeq \mathcal{Z}^0_n\cup \mathcal{Z}^1_n$. 
      
        Let $j\in\{0,1\}$ and let $z_0\in \mc{Z}^j\setminus \mc{Z}^{1-j}$. 
         Suppose that $\eps>0$ is so small  that 
         \begin{equation}\label{eqn:discdoesnotintersect} 
           \overline{D}(z_0,\eps)\cap\mc{Z}^{1-j}=\emptyset. \end{equation} By  Lemma~\ref{lem:separation}, 
         \begin{equation} \label{eqn:discexpansion}
             \overline{D}(z_n, 2^n\cdot \eps) \cap \mc{Z}^{1-j}_n = \emptyset,  \end{equation}
         where $z_n = \hat{F}^n(z_0)$.  
         If $2^n > 2\pi/\eps$, the disc in~\eqref{eqn:discexpansion} contains a straight line segment 
         of length $4\pi$ centred at $z_n$, which separates the strip $S$ into two parts, one to the left and one to the right of $z_n$. 
         In particular, $\mc{Z}^{1-j}$ is either entirely to the left or entirely to the right
         of $z_n$. 
         
   \begin{claim}[Claim 1]
     For sufficiently large $n_0$, and $j=0,1$, there are continua 
      $\tilde{\mathcal{Z}}_{n_0}^j\subset \mathcal{Z}_{n_0}^j\setminus\{\infty\}$ 
      such that  
       $\bigcap_{j=0}^1\tilde{\mathcal{Z}}_{n_0}^j\neq\emptyset$,
       such that
       $\tilde{\mathcal{Z}}_{n_0}^j \not\subset \tilde{\mathcal{Z}}_{n_0}^{1-j}$
       for $j=0,1$, and such that additionally 
        \begin{equation}\label{eqn:realparts} \re \hat{F}^n(z) \geq 2\pi \lvert s_{n+n_0}\rvert + \pi \end{equation}
       for all $n\geq 0$ and all $z\in \tilde{\mc{Z}}_{n_0}^0\cup \tilde{\mc{Z}}_{n_0}^1$. 
   \end{claim} 
   \begin{subproof} 
    Pick two points $z^0,z^1$ with 
       $z^j\in\mathcal{Z}^j\setminus (\mathcal{Z}^{1-j}\cup\{\infty\}$
       and let $\eps_1$ be so small that~\eqref{eqn:discdoesnotintersect} holds for both $z^0$ and $z^1$. 
       Let $\tilde{\mathcal{Z}}^j$ be a connected component of
       $\mathcal{Z}^j\setminus D(z^j,\eps_1)$ that contains a point of
       $\mathcal{Z}^0\cap \mathcal{Z}^1$. 
       By the boundary bumping theorem \cite[Theorem~5.6]{continuumtheory}, 
       $\tilde{\mathcal{Z}}^j$ intersects $\overline{D}(z^j,\eps_1)$, and 
        in particular $\tilde{\mathcal{Z}}^j \not\subset \mathcal{Z}^{1-j}$.
       
    Now pick a second point $\tilde{z}^j\in \mathcal{Z}^j\cap D(z^j,\eps_1)\setminus\{z^j\}$, for $j=0,1$. By Lemmas~\ref{lem:separation} and~\ref{lem:sector}, 
     there is $n_0$ such that 
      \begin{equation}\label{eqn:sectorinproof}
        \max(\re z^j_n, \re \tilde{z}^j_n) \geq \min(\lvert \im z^j_n\rvert , \lvert \im \tilde{z}^j_n\rvert) + 2\pi \geq 2\pi \lvert s_n\rvert + \pi, \end{equation}
      for $j =0,1$ and all $n\geq n_0$. We may additionally choose $n_0$ 
      so large that 
        \[ 2^{n_0} > \frac{2\pi}{\eps_1 - \max_j \lvert z^j - \tilde{z}^j}. \]
        If $z\in\tilde{\mc{Z}}^0\cup \tilde{\mc{Z}}^1$, then by~\eqref{eqn:discexpansion}, 
        for each $n\geq n_0$, 
       $\hat{F}^n(z)$ is to the right of 
      the left-hand side of~\eqref{eqn:sectorinproof} for either $j=0$ or $j=1$. 
       Similarly, $\hat{F}^n(z)$ is to the left of $\re z^0_n$ or of $\re z^1_n$.
      Hence the continua $\tilde{Z}_{n_0}^j \defeq F^{n_0}(\tilde{Z}^j)$ have the desired
      properties. 
   \end{subproof} 
  
  For simplicity of notation, 
   we assume in the following
   that $\mc{Z}^j$ were chosen to satisfy~\eqref{eqn:realparts} to begin with, 
   with $n_0=0$. 
%    We can now use Lemma~\ref{lem:Fandphi} to conclude that the inverse of $\hat{F}^n\colon \mc{C}\to \mc{C}_n$ maps real parts in a similar
%    way to the map $\phi^n$, where $\phi$ is the one-dimensional projection of $F$. 
   \begin{claim}[Claim 2]
      Let $z_0\in \mc{C}$ and set $z_k \defeq \hat{F}^{k}(z_0)\in \mc{C}_{n}$ for $k\geq 0$. If $0\leq k \leq n$, then 
       \[ \lvert \phi^{n-k}(\re z_{n})  - \re z_k\rvert \leq 12. \]
\end{claim}
     \begin{subproof}
       Set $r_k \defeq \phi^{n-k}(\re z_{n})$, $k=0,\dots, n$; we prove the claim by induction over $n-k$; the case $k=n$ is trivial. 
         Supppose the claim holds for some $k\in \{1,\dots,n\}$. 
          We have
         \[ \lvert r_{k-1} - \phi(\re z_k)\rvert = \lvert \phi(r_k) - \phi(\re z_k)\rvert \leq 6 \]
           by the inductive hypothesis and Lemma~\ref{lem:phiproperties}~\ref{item:phicontraction}. By~~\eqref{eqn:realparts} we also have
           \[ \lvert \im z_k\rvert \leq \lvert 2\pi \lvert s_{n}\rvert + \pi \leq \re z_k, \]
           so we may apply Lemma~\ref{lem:Fandphi} and obtain that
             \[ \lvert \phi(\re z_k) - \re z_{k-1}\rvert \leq 6. \]
             The claim follows for $k-1$, and the proof by induction is complete.
     \end{subproof}
       
    Choose $z^0,z^1$ with $z^j\in\mathcal{Z}^j\setminus \mathcal{D}$
       and let $\eps$ be again so small that~\eqref{eqn:discdoesnotintersect} holds for both $z^0$ and $z^1$.
       Define $z_n^j \defeq \hat{F}^n(z^j)$, and 
       fix $k$ so large that $2^k > (2\pi + 25 + 3K)/ \eps$.  
       Let $j\in \{0,1\}$ be such that $\re z_k^j < \re z_k^{1-j}$. 
       We define an integer quadruple $Q\defeq (A < B < C <D)$ by setting 
     \[ A \defeq \lceil \re z_k^j + 12 + K \rceil; \quad B \defeq A + K; \quad D \defeq \lfloor \re z_k^{1-j} - 12\rfloor; \quad C \defeq D - K. \] 
      By choice of $k$, we have 
          \begin{equation}\label{eqn:kseparation}   \re z_k^j + 25 + 3K \leq   \re \hat{F}^k(z) \end{equation}
         for all $z\in \mathcal{Z}^{1-j}$. Similarly,
         $\mc{Z}^j_k$ is to the left of
         $\re z_k^{1-j}-25-3K$. In particular, 
         $\lvert Q\rvert \geq K$, all 
         points in $\mathcal{Z}^{1-j}_k$ have real part greater than $B + 12$, and all points in $\mathcal{Z}^{j}_k$ have real part
      less than $C-12$. 
     
     Let $n_0$ be as in the hypothesis for the quadruple $Q$, and set $n \defeq k + n_0$. By Claim 2, and choice of $A$ and $B$, we have 
      \[ \phi^{n_0}(\re z_n^j) \leq \re z^j_k + 12 \leq A, \quad \phi^{n_0}(\re z_n^{1-j}) \geq \re z_k^{1-j} -12 \geq D. \]
      So if $I$ is the interval bounded by $\re z_n^j$ and $\re z_n^{1-j}$, then $\phi^{n_0}(I)\supset [A,D]$. 
      By hypothesis, this means that $I$ contains either two
      elements of $I\cap \phi^{-n_0}(B)$ that surround
      an element of $I\cap \phi^{-n_0}(C)$, or vice versa. 
      To fix our ideas, let us suppose the former, so there are $\zeta_1 < \omega < \zeta_2$ such that
      $\phi^{n_0}(\zeta_1)=\phi^{n_0}(\zeta_2) = B$, and $\phi^{n_0}(\omega)=C$. 
      (The opposite case is analogous.) 
      
      Let $\tilde{\zeta}_1$ be a point in $\mathcal{C}_{n}$ with real part $\zeta_1$, and similarly
       $\tilde{\zeta}_2$. Let $z_1$, $z_2$ be the corresponding preimages under $\hat{F}^{n_0}\colon \mc{C}_{n_0}\to \mc{C}_n$. 
         By Claim 2, 
         \[ \re z_1, \re z_2 \leq B + 12. \]
         Thus $z_1, z_2\in \mathcal{Z}_k^j$, and hence
         $\tilde{\zeta}_1,\tilde{\zeta}_2\in \mathcal{Z}_n^j$. 
         Since $\mathcal{Z}_n^j$ is a continuum, it 
         also contains a point $\tilde{\omega}$ with real part $\omega$. The 
         point $w\in \mathcal{Z}_k^j$ with $\hat{F}^{n_0}(w) = \omega$ satisfies
         $\re w \geq C - 12$ by Claim 2. This is a contradiction, since $\mathcal{Z}^j_k$ is to the left of $C-12$. 
           The proof of the theorem is complete.
      \end{proof}

\section{Mapping properties of the one-dimensional projection}\label{sec:projection}

Throughout this section, let $F\in\mc{H}$, and let $\phi$ be its one-dimensional projection. We introduce some notions and concepts that are 
  adapted from~\cite{hendersonpseudoarc}. 

\begin{defn}[Minimal maps to an interval]\label{defn:Un}
 Let $I = [A,D] \subset [6,\infty)$ be a closed interval, and let 
   $Q = (A < B < C < D)$ be a quadruple.
   
     Let $n\geq 0$. We 
      say that a closed interval $J \subset [4,\infty]$ is mapped \emph{minimally} over 
       $I$ by $\phi^n$ if
      $\phi^n(J)\supset I$, 
       and no smaller subinterval of $I$ has this property. 
       In this case, we also say that $J$ is mapped minimally over $Q$. 
     
   For $n\geq 0$, we define 
     \[ \U_n(Q,\phi) \defeq \U_n(I,\phi) \defeq  \{ J\subset [4,\infty] \colon J \text{ is mapped minimally to $I$ by $\phi^n$}\}. \]
     
    We also define 
     \[ \hat{\U}_n(Q,\phi) \defeq \{J\in \U_n(Q,\phi)\colon J \text{ is mapped crookedly over $Q$}\}. \]
\end{defn}

 \begin{lem}[Properties of the sets $\mc{U}_n$]\label{lem:Unproperties}
   Let $I = [A,D]$ and $Q$ be as in Definition~\ref{defn:Un}, and let $n\geq 0$. 
  \begin{enumerate}[(a)]
     \item If $\tilde{J}$ is a closed interval with $\phi^n(\tilde{J})\supset I$, then $\tilde{J}$ contains an element of $\mc{U}_n(I,\phi)$.\label{item:minimalsubset}
     \item If $J=[a,d]\in\mathcal{U}_n(I,\phi)$, then 
        $\phi^n(\{a,d\}) = \{A,D\}$ and $\{a,d\} = \phi^{-n}(\{A,D\})\cap \tilde{J}$.\label{item:Unendpoints}
     \item The interiors of intervals in $\mathcal{U}_n$ are pairwise disjoint.\label{item:Undisjoint}
    \item If $J\in\mathcal{U}_n(I,\phi)$, then $\phi^n(J) = I$.\label{item:mappingto}
    \item If $J\in \mathcal{U}_n(I,\phi)$, then $J\subset [6,\infty)$.\label{item:totheright}
     \item No interval $J\in\mathcal{U}_n(I,\phi)$ contains an interval of $\mathcal{U}_{\hat{n}}(I,\phi)$ for $\hat{n}\neq n$.\label{item:notnested}
     \item Each $\mathcal{U}_n$ is finite.\label{item:Unfinite}
     \item If $J\in \mathcal{U}_n(I,\phi)$, then $\phi^k(J)\in \mathcal{U}_{n-k}(I,\phi)$ and 
                 $J\in \mathcal{U}_k(\phi^k(J))$ for $k=0,\dots,n$.\label{item:Uncomposition} 
     \item In particular, for $k\leq n$, \label{item:Undecomposition}
        \[ \mathcal{U}_n(I,\phi) = \bigcup_{J\in \mathcal{U}_{n-k}(I,\phi)} \mathcal{U}_{k}(J,\phi). \]
     \item If $J\in \U_n(Q,\phi)$ and $\phi^k(J)\in \hat{\U}_{n-k}(Q,\phi)$ for some $k\leq n$, then $J\in\hat{\U}_n(Q,\phi)$.\label{item:crookedpropagation}
     \item If $\U_n(Q,\phi) = \hat{\U}_n(Q,\phi)$, then $\U_k(Q,\phi) = \hat{\U}_n(Q,\phi)$ for all $k\geq n$.\label{item:allcrooked}
  \end{enumerate}
 \end{lem}
 \begin{proof}
   Since we are only using a single function $\phi$ in this section, we suppress it from notation; i.e., we write $\mathcal{U}_n(J) \defeq \mathcal{U}_n(J,\phi)$. 
     We also simply write $\mathcal{U}_n = \mathcal{U}_n(I)$, for $I=[A,D]$, where this does not cause confusion.

    Claim~\ref{item:minimalsubset} 
     is a consequence of Zorn's lemma: consider the collection $\mathcal{J}$ of all
     intervals $J\subset \tilde{J}$ with $\phi^n(J)\supset I$. This
     collection is partially ordered by inclusion, and any descending chain 
     has a lower bound (its intersection), which is also in $\mathcal{J}$. 
     Hence there is a minimal element, which is in $\mathcal{U}_n$ by definition. 
   
   If $J\in\mathcal{U}_n$, then $\phi^{-n}(A)\cap J, \phi^{-n}(D)\cap J\neq \emptyset$. 
     Moreover, any subinterval of $J$ 
     bounded by an element of $\phi^{-n}(A)$ and an element of $\phi^{-n}(D)$ also maps over $I$. Hence the only points in $J$ that can map to
     $A$ and $D$ are the two endpoints, proving~\ref{item:Unendpoints}.   
     
 The next two claims follow immediately from~\ref{item:Unendpoints}. Indeed,   
   if two intervals intersect, but are not equal, then at least one has to contain
   one of the endpoints of the other. This never occurs for two intervals of
   $\mathcal{U}_n$, by~\ref{item:Unendpoints}, giving~\ref{item:Undisjoint}. 
   Similarly, if $J=[a,d]\in \mathcal{U}_n$, then
   $\phi^n\bigl((a,d)\bigr)$ is a connected subset of $\R\setminus \{A,D\}$ that contains
   $(A,D)$, and therefore $\phi^n\bigl((a,d)\bigr) = (A,D)$ and $\phi^n([a,d])=[A,D]$,
   as claimed in~\ref{item:mappingto}. 
   
  We have $\phi([4,6))\subset [4,6)$ by~Lemma~\ref{lem:phiproperties}~\ref{item:phidecreasing}. 
    If $J\in \mathcal{U}_n$ for some $n\geq 0$, then $\phi^n(J) = I \subset [6,\infty)$ 
    by~\ref{item:mappingto}, and therefore $J\cap [4,6)=\emptyset$. This 
    proves~\ref{item:totheright}. Claim~\ref{item:notnested} follows similarly
    from Lemma~\ref{lem:phiproperties}. Indeed, recall that $\phi(t)<t$ for $t\geq 6$.
    If $J\in\mathcal{U}_n$, then $\phi^n(J) = I$ by~\ref{item:mappingto}, and
    therefore $\phi^{\hat{n}}(J)$ to the left of $D$ for $\hat{n}>n$
    and to the right of $A$ for $\hat{n}<n$. Thus $J$ does not 
     map over $I$ under $\phi^{\hat{n}}$ for $n\neq \hat{n}$. 
    
     Next, 
       note that $\lim_{t\to\infty} \phi(t) = \infty$, and thus
         $\lim_{t\to\infty} \phi^n(t) = \infty$ for fixed $n$. 
         In particular, $\phi^{-n}(I) \supset \bigcup \mathcal{U}_n$ 
         is bounded for every $n$, and 
         $\delta \defeq \dist(\phi^{-n}(A),\phi^{-n}(D))>0$. 
         In conclusion, the intervals in $J\in\mathcal{U}_n$ have diameters bounded
         from below by $\delta$, 
         their interiors are disjoint, and their union is bounded. Hence 
        $\mathcal{U}_n$ is finite as claimed in~\ref{item:Unfinite}.
     
    Next, we prove~\ref{item:Uncomposition}. Let $J\in \mathcal{U}_n(I)$ and
     consider $J_k\defeq \phi^k(J)$. 
     Then $\phi^{n-k}(J_k) = \phi^n(J) = I$.  
     Clearly $J$ is mapped minimally to $J_k$ by $\phi^k$; otherwise,
     there would be a proper sub-interval also mapped to $J_k$ by $\phi^k$, 
     and hence to $I$ by $\phi^n$. So $J\in \mathcal{U}_{k}(J_k)$. 
     Similarly, suppose that $J_k$ was not mapped minimally by $\phi^{n-k}$. 
     By~\ref{item:minimalsubset}, there is a 
      proper sub-interval $\tilde{J}_k\subset J_k$ which belongs to
      $\mathcal{U}_{n-k}(I)$. But then 
        $J$ contains a proper sub-interval that is mapped over $\tilde{J}_k$, which
        would itself be mapped over $I$ by $\phi^{n-k}$.
         This is a contradiction to minimality of $J$, and proves that 
         $J_k\in \mathcal{U}_{n-k}(I)$. 

   The inclusion ``$\subset$'' in~\ref{item:Undecomposition} follows directly from~\ref{item:Uncomposition}. For the converse inclusion, let 
     $J\in \mathcal{U}_{n-k}(I)$ and $L\in \mathcal{U}_{k}(J)$. Then $\phi^n(L) = I$. So $L$ contains an element $\tilde{L}$ of $\mathcal{U}_n(I)$. 
     We have $\phi^k(\tilde{L}) \in \mathcal{U}_{n-k}(I)$ by~\ref{item:Uncomposition}
      and $\phi^k(\tilde{L})\subset \phi^k(L)=J$. So $\phi^k(\tilde{L})=J$. But also
     $\tilde{L}\in\mathcal{U}_k(\phi^k(\tilde{L})) = \mathcal{U}_k(J)$. So we must have $\tilde{L}=L$, as required. 
     
    Next suppose that $\phi^k(J)$ is mapped crookedly over $Q$; say the preimages $t_{B},\tilde{t}_B$ of $B$ surround
     a preimage $t_C$ of $C$. Then, by~\ref{item:minimalsubset}, 
       a subinterval $\tilde{J}$ of $J$ is mapped minimally over $[t_B,\tilde{t}_B]$ by $\phi^{k}$. In particular,
       $\tilde{J}$ contains a preimage of $t_C$ under $\phi^{k}$, which is an element of
       $\phi^{-n}(C)\cap J$. By~\ref{item:Unendpoints}, the subinterval $\tilde{J}$ 
       is bounded by a preimage of $t_B$ and a preimage of $\tilde{t}_B$ under $\phi^{k}$, both of which
       are elements of $\phi^{-n}(B)$. Hence $J$ is mapped crookedly by $f^n$, as claimed. 

     The final claim~\ref{item:allcrooked} is an immediate corollary of~\ref{item:crookedpropagation}. 
 \end{proof}

\section{Tracts and wiggles}\label{sec:tracts}
  In view of Theorem~\ref{thm:realisation} and Lemma~\ref{lem:Unproperties}~\ref{item:allcrooked}, our goal is to construct a function $F$ such that, for each of the countably
   many quadrilaterals $Q$ in Theorem~\ref{thm:realisation}, there is a number $n$ such that all members of $\U_n(Q,\phi)$ are mapped crookedly by $\phi^n$. 
   To do so, we begin with the strip $S$, and inductively introduce a 
      sequence of ``wiggles'' to the tract over increasing real parts; these wiggles allow us to deal with each quadrilateral $Q$ in turn. The desired
      model function $F$ is then obtained as the limit of these maps.

\begin{figure}
\begin{center}
\def\svgwidth{\textwidth}
%% Accompanies image file 'tracts.pdf' (pdf, eps, ps)
\begingroup%
  \makeatletter%
  \providecommand\color[2][]{%
    \errmessage{(Inkscape) Color is used for the text in Inkscape, but the package 'color.sty' is not loaded}%
    \renewcommand\color[2][]{}%
  }%
  \providecommand\transparent[1]{%
    \errmessage{(Inkscape) Transparency is used (non-zero) for the text in Inkscape, but the package 'transparent.sty' is not loaded}%
    \renewcommand\transparent[1]{}%
  }%
  \providecommand\rotatebox[2]{#2}%
  \newcommand*\fsize{\dimexpr\f@size pt\relax}%
  \newcommand*\lineheight[1]{\fontsize{\fsize}{#1\fsize}\selectfont}%
  \ifx\svgwidth\undefined%
    \setlength{\unitlength}{331.5811116bp}%
    \ifx\svgscale\undefined%
      \relax%
    \else%
      \setlength{\unitlength}{\unitlength * \real{\svgscale}}%
    \fi%
  \else%
    \setlength{\unitlength}{\svgwidth}%
  \fi%
  \global\let\svgwidth\undefined%
  \global\let\svgscale\undefined%
  \makeatother%
  \begin{picture}(1,0.17789678)%
    \lineheight{1}%
    \setlength\tabcolsep{0pt}%
    \put(0,0){\includegraphics[width=\unitlength,page=1]{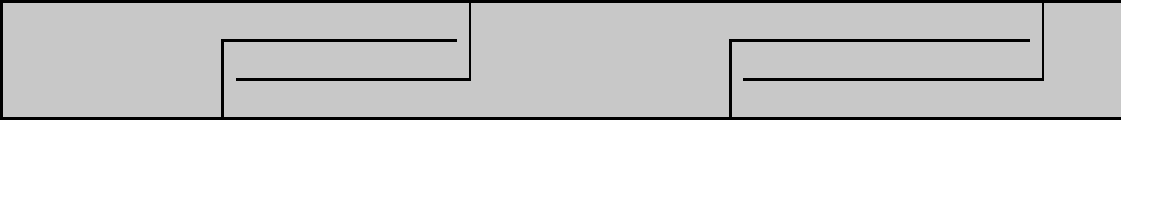}}%
    \put(0.18269152,0.00584763){\color[rgb]{0,0,0}\makebox(0,0)[lt]{\lineheight{1.25}\smash{\begin{tabular}[t]{l}$r_0$\end{tabular}}}}%
    \put(0,0){\includegraphics[width=\unitlength,page=2]{tracts.pdf}}%
    \put(0.39810901,0.00584763){\color[rgb]{0,0,0}\makebox(0,0)[lt]{\lineheight{1.25}\smash{\begin{tabular}[t]{l}$R_0$\end{tabular}}}}%
    \put(0.62429797,0.00584763){\color[rgb]{0,0,0}\makebox(0,0)[lt]{\lineheight{1.25}\smash{\begin{tabular}[t]{l}$r_1$\end{tabular}}}}%
    \put(0.89477499,0.00584763){\color[rgb]{0,0,0}\makebox(0,0)[lt]{\lineheight{1.25}\smash{\begin{tabular}[t]{l}$R_1$\end{tabular}}}}%
    \put(0.01331874,0.11677662){\color[rgb]{0,0,0}\makebox(0,0)[lt]{\lineheight{1.25}\smash{\begin{tabular}[t]{l}$T$\end{tabular}}}}%
  \end{picture}%
\endgroup%

\end{center}
\caption{The tracts $T$ used in our constrution, as defined in~\eqref{eqn:tracts}.\label{fig:tracts}}
\end{figure}
      
 More precisely, each of the functions involved in the construction is determined by specifying a tract $T\in\mathcal{T}$, which in turn
  is determined by sequences $(r_j)_{0\leq j < N}$ and $(R_j)_{0\leq j < N}$ of positive real numbers,
  where $N\leq \infty$, 
   $r_0 > 6$, $r_j > R_{j-1}+1 $ and $R_{j} > r_j + 2$ for all $j<N$. 
  The corresponding tract is then defined as follows (see Figure~\ref{fig:tracts}):
   \begin{align}\label{eqn:tracts} T \defeq \{x + iy &\colon 4 < x , \lvert y\rvert < \pi\}\, \setminus \\ \notag
        \bigcup_{0\leq j < N} \bigl( &\{r_j + iy \colon -\pi < y \leq \pi/3 \} \cup
                                         \{ x + \pi i /3 \colon r_j < x \leq R_j - 1 \}\, \cup \\ \notag
                                          &\{ R_j + iy \colon -\pi/3 \leq y < \pi \} \cup
                                          \{ x - \pi i /3 \colon r_j + 1 \leq x < R_j \} \bigr). 
    \end{align}
    
  Every tract $T$ is an element of $\mathcal{T}$ as defined in Definition~\ref{defn:classH}. We denote the
   subclass of $\mathcal{H}$ defined by the tracts as above by $\mathcal{F}$. If $F\in\mathcal{F}$, then 
   we write $T(F)$ for the tract of $F$, as before. We also write $r_j(F)$ and $R_j(F)$ for the corresponding
   sequences as above. 
      For inductive purposes, we use the convention that $R_{-1}(F) = 5$.

   We also write $N(F) = N \leq \infty$ for the number of ``wiggles'' occurring in the tract of $F$,
   and denote by $\mathcal{F}_N$ the subclass consisting of maps $F\in\mathcal{F}$ with $F(N)=N < \infty$.
   In the following, when $F,\tilde{F}, F_j$ etc.\ are members of $\mathcal{F}$, we always take $\phi$, $\tilde{\phi}$, $\phi_j$ etc. to denote their one-dimensional projections.

\begin{prop}[Bounded decorations in $\F$]\label{prop:bounded_decorations}
 There exists a universal constant $\nu_0>0$ such that $\mathcal{F}\subset\mathcal{H}_{\nu_0}$. 
\end{prop}
\begin{proof}
  Let $F\in\mathcal{F}$ and $T = T(F)$. The 
    set $\partial T\setminus\{4\}$ has two connected component, one containing
    $[4,\infty) + \pi i$ and one containing $[4,\infty) - \pi i$. We call these the \emph{upper} and \emph{lower}
     boundaries of $T$, respectively. 
     By the definition of $\mathcal{F}$, every $z\in T$ can be connected both to the upper and to the
     lower boundary by a straight vertical segment $\alpha_z\subset T$ of length at most $2\pi$. 

  \begin{claim}
    There is a universal constant $C$ with the following property. If $F\in\mathcal{F}$ and $z\in T=T(F)$, there are 
      half-open arcs
     $\beta^+_z, \beta^-_z\subset T$ such that 
       \begin{enumerate}[(a)]
          \item Each $\beta^+$ and $\beta^-$ connects $z$ to an endpoint on $\partial T$, 
          \item $\beta^+$ and $\beta^-$ have Euclidean diameter at most $C$;
          \item $F(\beta_z^{+})$ connects $F(z)$ to an endpoint on $i[0,\infty)$, and 
                          $F(\beta^-_z)$ connects $F(z)$ to a point of $i(-\infty,0]$. 
      \end{enumerate}
  \end{claim}
  \begin{subproof}
     The harmonic measure in $\HH$ of $i [0,\infty)$, seen from $5$, is $1/2$, as is that of
     $i (-\infty,0]$. 
        (See e.g.\ \cite{garnettmarshall}%%%\cite[???]{pommerenke} 
        for the definition of harmonic measure.) 
         Furthermore, $\dist(5,\partial T) = 1$ by construction. 
        It follows from~\cite[Corollary~4.18]{pommerenke} that
       there is a universal constant $K_0$ with the following property: there exist a hyperbolic geodesics $\gamma^+$ 
         and $\gamma^-$
       of $T$ that have diameter at most $K_0$, and such that $F(\gamma^{\pm})$ connects $5$ to $\pm i[0,\infty)$. 
       We set $C \defeq K_0 + 2\pi + 1$.

     Let $\theta\in\R$ be the unique number with $F^{-1}(i\theta) = 4$. To fix our ideas, suppose that
       $\theta\leq 0$; the case $\theta\geq 0$ is analogous. Let $z\in T$. The curve $F(\alpha_z)$ has endpoints
         $i\theta^+$ and $i\theta^-$ with $\theta^- < \theta < \theta^+$. In particular,
         we can take $\beta^-$ to be the piece of $\alpha_z$ connecting $z$ to the lower boundary. 
          If $\theta^+ \geq 0$, we can do the same for the upper boundary, and are done. 
         
     Otherwise, the segment $\alpha_z$ intersects $\gamma \cup (4,5]$. We can thus connect $z$ to a point $z'$ of 
       $\gamma \cup (4,5]$, and connect $z'$ to the endpoint of $\gamma$ on $\partial T$ by a subcurve
       of $\gamma\cup (4,5]$. Overall, this curve has diameter at most $2\pi + K_0 + 1=C$, as desired.
  \end{subproof}

  To complete the proof, we proceed as in      \cite[Proposition~7.4]{arclike}. By \cite[Lemma~4.21]{pommerenke},
    there is a universal constant $K_1$ with the following property. Let $x\in\R$, let 
    $\beta$ be the geodesic of $\HH$ connecting $x$ to $ix$, and let $\tilde{\beta}\subset\HH$ be any curve
    connecting $ix$ to $[0,\infty)$. Then 
      \[ \diam(F^{-1}(\beta))\leq K_1\cdot \diam(F^{-1}(\tilde{\beta})). \] 
      
    The diameter $\diam(\tilde{\beta})$ can be chosen less than $C + \eps$, for any $\eps>0$. Indeed,
      we may 
      choose $\delta$ small enough that $F^{-1}( [0,\delta] + ix)$ has diameter less than $\eps$, and then 
      connect $z\defeq F^{-1}(\delta + ix)$ to $F^{-1}([0,\infty))$ using the curve $\beta^+_z$ if $x<0$, and 
        using $\beta^-_z$ if $x>0$. 
    
       So $\diam(F^{-1}(\beta)) \leq K_1 \cdot C$, and the claim follows if we take
        $\nu\defeq 2\cdot K_1\cdot C$.
\end{proof} 

\begin{figure}
\begin{center}
\def\svgwidth{\textwidth}
%% Accompanies image file 'tract_curve_alpha.pdf' (pdf, eps, ps)
\begingroup%
  \makeatletter%
  \providecommand\color[2][]{%
    \errmessage{(Inkscape) Color is used for the text in Inkscape, but the package 'color.sty' is not loaded}%
    \renewcommand\color[2][]{}%
  }%
  \providecommand\transparent[1]{%
    \errmessage{(Inkscape) Transparency is used (non-zero) for the text in Inkscape, but the package 'transparent.sty' is not loaded}%
    \renewcommand\transparent[1]{}%
  }%
  \providecommand\rotatebox[2]{#2}%
  \newcommand*\fsize{\dimexpr\f@size pt\relax}%
  \newcommand*\lineheight[1]{\fontsize{\fsize}{#1\fsize}\selectfont}%
  \ifx\svgwidth\undefined%
    \setlength{\unitlength}{322.87501037bp}%
    \ifx\svgscale\undefined%
      \relax%
    \else%
      \setlength{\unitlength}{\unitlength * \real{\svgscale}}%
    \fi%
  \else%
    \setlength{\unitlength}{\svgwidth}%
  \fi%
  \global\let\svgwidth\undefined%
  \global\let\svgscale\undefined%
  \makeatother%
  \begin{picture}(1,0.10685249)%
    \lineheight{1}%
    \setlength\tabcolsep{0pt}%
    \put(0,0){\includegraphics[width=\unitlength,page=1]{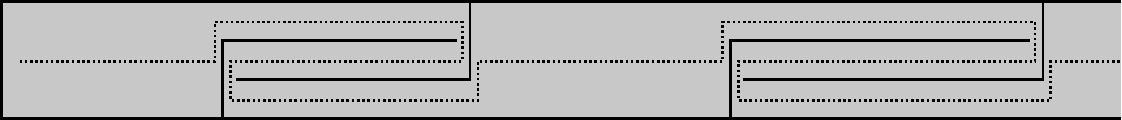}}%
    \put(0.02439024,0.05923346){\color[rgb]{0,0,0}\makebox(0,0)[lt]{\lineheight{1.25}\smash{\begin{tabular}[t]{l}$\alpha$\end{tabular}}}}%
  \end{picture}%
\endgroup%

\end{center}
\caption{The curve $\alpha$ stays distance $1/2$ away from $\partial T$.\label{fig:alpha}}
\end{figure}

\begin{prop}[Growth of functions in $\F$]\label{prop:Fgrowth}
 There is a constant $C>1$ with the following property. Let $F\in\mathcal{F}$, and let $z\in T(F)$ with $\lvert F(z)\rvert \geq 4$. 
   If $z$ does not belong to the two bottom thirds of any ``wiggle'', i.e., if 
      \[ z\notin W_j(F) \defeq \{\zeta \in T(F)\colon r_j < \re \zeta < R_j \text{ and } \im \zeta < \pi/3 \} \]
      for all $j<N(F)$, 
      then
        \[ \frac{\re z}{C} \leq \log \lvert F(z)\rvert \leq C\re z. \]
      If $z\in W_j(F)$, then 
         \[ \frac{R_j}{C} \leq \log \lvert F(z)\rvert \leq C\cdot R_j. \] 
\end{prop}
\begin{proof}
   Note that there is an arc $\alpha$ that connects $5$ to $\infty$ in $T$
     and remains at distance at most $1/2$ from $\partial T$, namely 
     \begin{align*} \alpha \defeq \left[5, r_0 - \frac{1}{2}\right] \cup 
           \bigcup_{j\geq 0} &\left(r_j - \frac{1}{2} + i[0,2\pi/3] \right)
                \cup \left(\left [r_j - \frac{1}{2} , R_j - \frac{1}{2}\right] + 2\pi i /3 \right) \\ &\cup
                   \left( R_j - \frac{1}{2} + i[0,2\pi/3] \right) \cup
                   \left[r_j + \frac{1}{2}, R_j - \frac{1}{2}\right] \\ &\cup
                   \left( r_j + \frac{1}{2} + i [-2\pi/3,0]  \right)  \cup
                   \left( \left[ r_j + \frac{1}{2} , R_j + \frac{1}{2} \right] - 2\pi i/3\right) \\ &\cup
                   \left(R_j + \frac{1}{2} + i[0,-2\pi/3]\right) \cup \left[R_j + \frac{1}{2} , r_{j+1} - \frac{1}{2} \right]. \end{align*} 
    (See Figure~\ref{fig:alpha}.) Let $z\in \alpha$, and let $\alpha_z$ be the part of $\alpha$ that connects
    $5$ to $z$. Let us estimate the Euclidean length $\ell(\alpha_z)$, first when $z\notin W_j(F)$ for any $j$. Let 
    $j$ be maximal such that $\re z > R_j$. Recall that $R_j > 7 + 3j$. We have 
      \begin{align*} \ell(\alpha_z) &\leq \re z - 5 + \pi + \sum_{i\leq j} \left(2\pi + R_j - r_j - 1\right) \\ &\leq
                                     3\re z + (2j+3)\pi < 3 \re z + R_j \pi < (3 + \pi)\cdot \re z  \end{align*}
    If $z\in W_j(F)$, then similarly 
      \[ \ell(\alpha_z) \leq (3 + \pi) \cdot R_j. \]
    On the other hand, clearly any curve connecting $5$ to a point $z\in T$ must have Euclidean length at least
       $\re z -5$, and at least $R_j$ if $z\in W_j$. 

   Now let $z$ be as in the statement of the proposition, and let 
     \[ \gamma \defeq F^{-1}(\{\zeta\in \HH\colon \lvert \zeta \rvert = \lvert F(z)\rvert \}) \]
     be the vertical geodesic containing $z$, and set $\tilde{z}\defeq F^{-1}(\lvert F(z)\rvert)\in \gamma$. By Proposition~\ref{prop:bounded_decorations}, 
       the Euclidean diameter of $\gamma$ is bounded from above by $\nu_0$. 
       Define 
          \[ R \defeq \begin{cases}
                                                        \re z & \text{if $z\notin W_j$ for all $j$}; \\
                                                        R_j & \text{if $z\in W_j$}.\end{cases} \]
       By the above, and the standard estimates on the hyperbolic metric on $T$, we have 
       \[ R - \nu_0 - 5 \leq \dist_T(5,\gamma) \leq 
           (3+\pi) \cdot (R + \nu_0) \leq (3+\pi + \nu_0/4) \cdot R. \] 
      Moreover, 
        \[ \dist_T(5,\gamma) = \dist_T(5,\tilde{z} = \dist_{\HH}(5,\lvert F(z)\rvert) = \log \frac{\lvert F(z)\rvert}{5}. \]
     When $R> 2(\nu_0 + 5)$, say, the claim follows; we can ensure that it holds for smaller $R$ also by choosing
      $C$ sufficiently large. (Recall that $R, \lvert F(z)\rvert \geq 4$.) 
\end{proof}

\begin{prop}[Number of intervals increases only through wiggles]\label{prop:Un1}
  Let $F\in\mathcal{F}$. Let $I=[A,D]$ be an interval with $A\geq 6$ and $\lvert I\rvert \geq \nu_0$, where $\nu_0$ is the constant from Proposition~\ref{prop:bounded_decorations}. Suppose 
   that furthermore 
    \[ I \not\subset [r_k(F) - \nu_0 , R_k(F) + \nu_0 ] \]
   for any $k\geq 0$. Then 
     \[ \# \U_1(I,\phi) = 1. \]
\end{prop}
\begin{proof}
  The hypothesis means that there is $\tau\in[A,D]$ such that 
     \[ R_k + \nu_0 < \tau < r_{k+1} - \nu_0 \]
     for some $k$. 

  In particular, $T$ intersects the vertical lines at real parts 
    $\tau$ and $\tau + \nu_0$ each in exactly one vertical line segment. 
   Let $r$ be maximal such that the vertical geodesic 
      \[ \gamma = F^{-1}(\{z\in\HH\colon \lvert z\rvert = r \}). \]
     contains a point at real part $\tau$. Then, by Proposition~\ref{prop:bounded_decorations}, 
       $\gamma$ separates all points of $T$ at real part
       less than $\tau$ from all points at real part greater than $\tau + \nu_0$. 
        
   If $J\in\U_1(I,\phi)$, then $F^{-1}(J)$ connects real parts $\tau$ and $\tau+\nu_0$, and must hence intersect 
      $\gamma$, which means
      that $r \in J$.  By Lemma~\ref{lem:Unproperties}~\ref{item:Undisjoint}, two different
       intervals in $\U_1(I,\phi)$ cannot both contain the same point $r$. This proves the claim.
       \end{proof}

\section{Approximation properties of the construction}

As already mentioned, the positions $(r_n,R_n)$ of the ``wiggles'' in our example will be defined inductively. A key fact  is that, if $F\in\mathcal{F}_N$, and the next wiggle $(r_N,R_N)$ is chosen sufficiently far to the right, then the inverse of the 
 next function $\tilde{F}$ will be close to $F^{-1}$, at least up until the end of the $n$-th ``wiggle'' of $T(\tilde{F})$. 
 In particular, key properties of the one-dimensional projection $\phi$ of $F$ will be preserved under this approximation. 
 
To make this statement precise, let us say
  that $\tilde{F}$ is \emph{$(N,\rho)$-close} to a function $F\in\mathcal{F}_N$ if 
 \begin{enumerate}[(a)]
   \item $N(\tilde{F})>N$;
   \item $r_n(\tilde{F}) = r_n(F)$ and $R_n(\tilde{F}) = R_n(F)$ for $n < N$;\label{item:samewiggles}
   \item $r_N(\tilde{F})\geq \rho$. 
  \end{enumerate} 
  We begin by noting that $(N,\rho)$-close tracts, for large $\rho$, are also close in the
  sene of Carath\'eodory kernel convergence. (See~\cite[Section~1.4]{pommerenke} for a discussion of kernel convergence.)
  
\begin{lem}[Carath\'eodory convergence]\label{lem:caratheodory}
  Let $F\in \mathcal{F}_N$. Suppose that $(F_j)_{j=0}^{\infty}$ is a sequence of functions in $\mathcal{F}$, such that
    $F_j$ is $(N,\rho_j)$-close to $F$, with $\rho_j\to\infty$. Then $T(F_j)$ converges to $T(F)$, with respect to any point
    $z_0\in T$, in the sense of Carath\'eodory kernel convergence. 
\end{lem}  
\begin{proof}
  Let $K$ be a compact, connected set. We must show that the following are equivalent: 
   \begin{enumerate}[(i)]
     \item $K\subset T(F)$;\label{item:KinT}
     \item $K\subset T(F_j)$ for all but at most finitely many $j$. \label{KinTn}
   \end{enumerate} 
   By condition~\ref{item:samewiggles} in the definition of $(N,\rho)$-closeness, we have 
      $T(F_j)\subset T(F)$ and 
       \[ T(F) \cap \{ \re z < \rho_j\} \subset T(F_j) \] 
    for all $j$. So if 
      $j$ is sufficiently large that $\re z < \rho_j$ for all $z\in K$, then $K\subset T(F_j)$ if and only if $K\subset T(F)$. 
\end{proof}   

\begin{prop}[Convergence of inverse functions]\label{prop:caratheodory1}
 Let $F\in \mathcal{H}_{\nu}$, let $\nu>0$, and suppose that $(F_j)_{j=0}^{\infty}$ is a sequence of functions in $\mathcal{H}_{\nu}$ such that
   $T_j(F)$ converges $T(F)$ with respect to the point $5$, in the sense of Carath\'eodory kernel convergence. 
   
   Then $F_j^{-1} \to F^{-1}$ locally uniformly on $\HH$. 
\end{prop} 
\begin{proof}
 We write $T = T(F)$ and $T_j = F(T_j)$. 
    Let $\psi\colon \DD\to T$ and $\psi_j\colon \DD\to T_j$ be conformal isomorphisms, normalised such that
        $\psi(5)=\psi_j(5) = 5$ and $\psi'(5),\psi_j'(5)>0$. By Carath\'eodory's kernel convergence theorem,
        $\psi_j\to \psi$ locally uniformly on $\DD$. Recall that $\psi$ extends continuously to $\partial\DD$ by the Carath\'eodory-Torhorst theorem, 
        taking the value $\infty$ at exactly one point $\psi^{-1}(\infty)$ of $\partial\DD$. 
        We have 
         \[ F^{-1} = \psi \circ M, \]
         where $M\colon \DD\to\HH$ is the M\"obius transformation with $M(5)=0$ and $M(\infty)=\psi^{-1}(\infty)$. The analogous statements hold for
         $\psi_j$ and $F_j$. So it suffices to prove that $\psi_j^{-1}(\infty)\to\psi^{-1}(\infty)$. 
         
      Let $\delta>0$. Let $R$ be so large that the set of points
        at real part $\geq R $ in the boundary of $\tilde{S} = \{ a + ib\colon -\pi<b<\pi\}$ has harmonic measure at most $\delta/(2\pi)$, seen from $5$. Since
        $T_j\subset \tilde{S}$, the set of points at real part $\geq R$ in $\partial T_j$ also has harmonic measure at most $\delta/(2\pi)$ from $5$, independently of $j$. 
        
       Choose a point $\zeta\in T$ such that  $\re \zeta > R + \nu$ and 
         $\lvert \psi^{-1}(\zeta) - \psi^{-1}(\infty)\rvert < \delta/4$. For sufficiently large $j$, we have 
           \[ \lvert \psi_j^{-1}(\zeta) - \psi^{-1}(\zeta)\rvert < \frac{\delta}{4}. \] 
        Now let $\gamma_j$ be the geodesic of $\DD$ through $\psi_j^{-1}(\zeta)$ which is perpendicular to the radius connecting $0$ and $\psi_j^{-1}(\infty)$. Then
         $\psi_j(\gamma_j)$ is the vertical geodesic of $T_j$ through $\zeta$, and hence has diameter at most $\nu$. By choice
         of $\zeta$, 
           all real parts of $\psi_j(\gamma_j)$ are greater than $R$, and the harmonic measure of the arc of $\partial\DD$ separated from $0$ by $\gamma$ is
           at most $\delta/(2\pi)$. It follows that this arc has length at most $\delta/2$ and 
             $\lvert \psi_j^{-1}(\zeta) - \psi_j^{-1}(\infty)\rvert \leq \delta/2$. We conclude that
              \[ \lvert \psi^{-1}(\infty) - \psi_j^{-1}(\infty)\rvert \leq \lvert \psi^{-1}(\infty) - \psi^{-1}(\zeta)\rvert + 
                            \lvert \psi^{-1}(\zeta) - \psi_j^{-1}(\zeta)\rvert + \lvert \psi_j^{-1}(\zeta) - \psi_j^{-1}(\infty)\rvert < \delta. \]
     Since $\delta>0$ was arbitrary, the proof is complete.                                         
\end{proof} 
\begin{rmk}
  It is not sufficient to require that
     $F_n\in\mathcal{H}$ (rather than $F_n\in\mathcal{H}_{\nu}$). 
   Indeed, suppose that $T=S = \{  x + iy\colon x > 4\text{ and } \lvert y \rvert < \pi\}$ and 
      \[ T_n = \{ x+ iy\in S\colon x  < n + 6 \} \cup \{x + iy \in S\colon y > \pi - 1/n \} \cup
                    \{ x + (\pi - 1/n)i \colon 4<x<1/n\}. \]
    Then $T_n\to T$ in kernel, and in fact the boundaries of $T_N$ in $\Ch$ converge to that of $T$ in the Hausdorff metric.
     However, let $t_n>5$ be minimal with $\im F_n^{-1}(t_n) = \pi - 1/n$; then 
     $\dist_{T_n}(5,t_n)\to \infty$, and thus $t_n\to \infty$. However, by the 
     Gehring-Hayman theorem \cite[Theorem~4.20]{pommerenke}, $F_n^{-1}([5,t_n])$ remains uniformly bounded as $n\to\infty$.
     So $F_n^{-1}$ does not converge locally uniformly to $F^{-1}$ on $\HH$. 
%     Compare Figure~\ref{fig:caratheodory}.
\end{rmk} 
  
 Now we are able to prove the main approximation result of this section. Note that  Proposition~\ref{prop:caratheodory1} 
  shows that $\tilde{\phi}$ is close to $\phi$ whenever $\tilde{F}$ is $(N,\rho)$-close to $F$, with $\rho$ sufficiently large. 
  The following proposition sharpens this further by showing that $\tilde{\phi}^n(t)$ is close to $\phi^n(t)$, for all $n$,
  provided that these values belong to a bounded subinterval of $[4,\infty)$, and furthermore $t$ itself 
  is not too much larger than the position of $R_N(\tilde{F})$. Note that the latter means that $t$ is allowed to be
  large when $\rho$ is large.
  
\begin{prop}[Approximation by $(n,\rho)$-close functions]\label{prop:caratheodory}
  Suppose that $F\in \mathcal{F}_N$. Let $\eps>0$ and $\tau \geq 5$. Then there is a number $\rho > R_{N-1}(F)$ with the following property. 
  
  Suppose that $\tilde{F}\in \mathcal{F}$ is $(N,\rho)$-close to $F$. If $4\leq t \leq 2R_{N}(F)$ and
    $\min(\phi^n(t), \tilde{\phi}^n(t)) \leq \tau$ for some $n\geq 0$, then 
    \begin{equation}\label{eqn:phicaratheodory} \lvert \phi^n(t) - \tilde{\phi}^n(t)\rvert \leq \eps. \end{equation}
\end{prop}
\begin{proof}
%  Let $(F_j)_{j=1}^{\infty}$ be a sequence of functions with $N(F_j)>N$ for all $j$ and $r_n(F_j)\to\infty$. We have to show that 
%    the conclusion of the proposition holds for $F_j$ when $j$ is sufficiently large. 
%    Set $T_j \defeq T(F_j)$. 
   By Lemma~\ref{lem:caratheodory} and Proposition~\ref{prop:caratheodory1}, 
     $F_j^{-1}\to F^{-1}$ locally uniformly on $\HH$. In particular, $\phi_j\to\phi$ uniformly on any compact subset of $(0,\infty)$. 
 We first  prove~\eqref{eqn:phicaratheodory} when $t$ belongs to a fixed interval of the real line.
  \begin{claim}[Claim 1]
    Let $\tilde{\tau}>4$ and $\delta>0$. Then there is $\rho(\tilde{\tau},\delta) > R_{N-1}(F)$ 
      with the following property. Suppose that           $\tilde{F}$ is $(N,\rho(\tilde{\tau},\delta))$-close to $F$. 
        If $t_1,t_2 \in [4,\tilde{\tau}]$ with $\lvert t_1 - t_2\rvert \leq \delta$, then 
             \[ \lvert \phi^n(t_1) - \tilde{\phi}^n(t_2)\rvert \leq \delta\] 
              for all $n\geq 0$. 
  \end{claim}
  \begin{subproof} 
   We may assume without loss of generality that $\tilde{\tau}\geq 6$, so that $\phi([4,\tilde{\tau}])\subset [4,\tilde{\tau}]$. 
         By Lemma~\ref{lem:caratheodory} and Proposition~\ref{prop:caratheodory1},  
         there is $\rho(\tilde{\tau},\delta)$ such that, under the hypotheses of the claim, 
         $\phi$ and $\tilde{\phi}$ are closer than $\delta/2$ on $[4,\tilde{\tau}]$. 

   The claim follows by induction from the case $n=1$, which follows from the contraction properties of $\phi$: 
     suppose that the claim holds
     for $n$, and let $t\leq \tilde{\tau}$. We have 
   \[ \lvert \phi(t_1) - \tilde{\phi}(t_2)\rvert \leq 
      \lvert \phi(t_1) - \phi(t_2)\rvert + \lvert \phi(t_2) - \tilde{\phi}(t_2)\rvert \leq 
      \frac{\lvert t_1 - t_2 \rvert}{2} + \frac{\delta}{2} \leq \delta \]
       by the induction hypothesis and Lemma~\ref{lem:phiproperties}.
  \end{subproof}

   Now let $\tilde{F}$ be as in the hypotheses, for $\rho$ sufficiently large (to be specified below). 
     Set $\tilde{T}\defeq T(\tilde{F})$ and $R_{N-1}\defeq R_{N-1}(F) = R_{N-1}(\tilde{F})$. 
  Let $C$ be the universal 
    constant from Proposition~\ref{prop:Fgrowth}. For $t\geq \exp(C\cdot R_{N-1})$, the point $z=F^{-1}(t)$ does not belong to the bottom two thirds of any wiggle
    of $T$, and hence 
     \[ \phi(t)/C \leq \log t \leq C\cdot \phi(t). \] 
   Moreover, if additionally $t\leq \exp(R_N(\tilde{F})/C)$, then $\tilde{\phi}(t)$ does not belong to the bottom two thirds of any wiggle of $\tilde{T}$, and the same
   estimate holds for $\tilde{\phi}$. In particular, both values are comparable up to a multiplicative error of at most $C^2$. 
   
    \begin{claim}[Claim 2]
      Set $\hat{C}\defeq \max(6,4C^2/3)$. 
       Suppose that $t \leq \exp(R_N(\tilde{F})/C)$, and that 
          \[ \max( \phi^n(t),\tilde{\phi}^n(t)) \geq C^2\cdot R_{N-1}. \] 
         Then
             \[  \frac{1}{\hat{C}} \leq \frac{\phi^n(t)}{\tilde{\phi}^n(t)} \leq \hat{C}.  \] 
    \end{claim}
    \begin{subproof}
      We prove the claim by induction; it is trivial for $n=0$. Now suppose that it holds for $n$. 
        Let us suppose that 
        $\phi^{n+1}(t) \geq C^2\cdot R_{N-1}$; the alternative case is analogous, exchanging the roles
         of $\phi$ and $\tilde{\phi}$. The inductive hypothesis implies 
        \[ \dist_{\HH}(\phi^n(t),\tilde{\phi}^n(t)) \leq \log\hat{C} + \frac{\pi}{2}. \]
       Therefore 
         \[ \lvert \tilde{\phi}(\phi^n(t)) - \tilde{\phi}^{n+1}(t)\rvert \leq 
               2 \dist_{T}(\tilde{\phi}(\phi^n(t)) , \tilde{\phi}^{n+1}(t))
                      \leq 2\log \hat{C} + \pi < \hat{C}, \]
     since $\hat{C} \geq 6$. Furthermore, the assumption on $\phi^{n+1}(t)$ means that
       $\phi^n(t)\geq \exp(C\cdot R_{N-1})$, by Proposition~\ref{prop:Fgrowth}. By the above, 
       \[ \phi^{n+1}(t)/C^2 \leq \tilde{\phi}(\phi^n(t)) \leq C^2 \cdot \phi^{n+1}(t). \]
       In particular, 
          \[ \phi^{n+1}(t) \leq C^2 \tilde{\phi}(\phi^n(t)) \leq C^2(\tilde{\phi}^n(t) + \hat{C}) \leq
                (C^2 + \hat{C}/4) \cdot \tilde{\phi}^n(t) \leq \hat{C}\cdot \tilde{\phi}^n(t). \]
      The opposite inequality follows analogously. 
    \end{subproof}

 We may suppose that \[ \tau\geq \max( 4\hat{C}/\eps, C^2\cdot R_{N-1}). \] Define $\tilde{\tau}\defeq \hat{C}\cdot \exp(C\cdot \exp(C\cdot \tau))$. 
  Then, by Proposition~\ref{prop:Fgrowth}, $\phi^2(t) > \tau/\hat{C}$ when $t> \tilde{\tau}$. The same
   holds for $\tilde{\phi}$, assuming that we have chosen $\rho$ larger than $\tilde{\tau}$. 
   Now choose 
      \[ \rho \geq \max( , \rho(\tilde{\tau},\delta) , \rho(\tilde{\tau},\eps) ), \]
       where $\rho(\tilde{\tau},\cdot)$ is as in Claim 1, with $\delta = \min(\hat{C},\eps/2)$. Note that
   $\tilde{\tau}$ depends only on $R_{N-1}$ and $\tau$, 
   and thus $\rho$ depends only on $F$, $\eps$ and $\tau$, as required. 
     
  Suppose that $t$ is as in the statement of the proposition. It is enough
   to prove the claim when $n$ is minimal with the stated property, since it then follow for larger $n$ 
   by Claim~1. Moreover, the claim holds for $n\leq 2$ by choice of $\tilde{\tau}$ and Claim 1. 

   So now suppose that $n\geq 2$. 
      We have $\phi^j(t),\tilde{\phi}^j(t) > \tau$ for $j < n$, and furthermore 
       either $\phi^{n-2}(t)\leq \tilde{\tau}/\hat{C}$ or 
       $\tilde{\phi}^{n-2}(t)\leq \tilde{\tau}/\hat{C}$. 
    By Claim~2, $\phi^{n-2}(t)$ and $\tilde{\phi}^{n-2}(t)$ are comparable by a factor of $\hat{C}$. 
    In particular both (and their images) are less than
     $\tilde{\tau}$. We now proceed similarly as in the proof of Claim~2. Firstly, again 
       \[ \lvert \phi^{n-1}(t) - \phi(\tilde{\phi}^{n-2}(t))\rvert \leq 2 \log \hat{C} + \pi < \hat{C}.   \] 
    By choice of $\tilde{\tau}$,
      \[ \lvert \phi(\tilde{\phi}^{n-2}(t)) - \tilde{\phi}^{n-1}(t)\rvert \leq \delta, \]
    and thus 
      \[ \lvert \phi^{n-1}(t) - \tilde{\phi}^{n-1}(t)\rvert \leq \hat{C} + \delta \leq 2\hat{C}. \] 
      
     In particular, we have 
       \[\dist_{\HH}(\phi^{n-1}(t) , \tilde{\phi}^{n-1}(t)) \leq \frac{2\hat{C}}{\tau} \leq  \delta/2.\]
      Arguing as above, we conclude that 
        \[ \lvert \phi^n(t) - \tilde{\phi}^n(t)\rvert \leq
           \lvert \phi^n(t) - \phi(\tilde{\phi}^{n-1}(t))\rvert + \lvert \phi(\tilde{\phi}^{n-1}(t)) - \tilde{\phi}^{n-1}(t)\rvert 
             \leq \delta + \delta = \eps. \qedhere \]
\end{proof}

\begin{defn}[Larger quadruples]
  If $Q = (A < B < C < D)$ and $\tilde{Q} = (\tilde{A} < \tilde{B} < \tilde{C} < \tilde{D})$ are quadruples, we write
     $Q < \tilde{Q}$ if 
        \[ \tilde{A} < A < B < \tilde{B} < \tilde{C} < C < D < \tilde{D}. \]
\end{defn}

\begin{prop}[$\U_n$ and approximation]\label{prop:Qcaratheodory}
  Let $F\in\mathcal{F}_N$, and let $Q < \tilde{Q}$ be quadruples. Let $n_0\geq 0$. 
  
  Then there is $\rho> R_{N-1}(F)$ with the following property. If $\tilde{F}\in \mathcal{F}$ is $(N,\rho)$-close to $F$, then 
  \begin{enumerate}[(a)]
    \item For $n\leq n_0$, all elements of $\U_n(\tilde{Q},\tilde{\phi})$ are to the left of $\rho$.\label{item:intervalstotheleft}
    \item For every $n\geq 0$ and every $\tilde{J}\in \U_n(\tilde{Q},\tilde{\phi})$ that is to the left of $2 R_N(\tilde{F})$, 
       there is $J\in \U_n(Q,\phi)$ with $J\subset \tilde{J}$.\label{item:Unclose}
    \item If $J$ as in~\ref{item:Unclose} is mapped crookedly over $Q$ by $\phi^n$, then $\tilde{J}$ is mapped crookedly over $\tilde{Q}$ by $\tilde{\phi}^n$.
   \end{enumerate}
  
  In particular, for $n\leq n_0$,   
          $\#\U_n(\tilde{Q},\tilde{\phi}) \leq \U_n(Q,\phi)$ and 
           \[ \# \U_n(\tilde{Q},\tilde{\phi}) - \# \hat{\U}_n(\tilde{Q},\tilde{\phi}) \leq \# \U_n(Q,\phi) - \# \hat{\U}_n(Q,\phi). \] 
\end{prop}
\begin{proof}
  That we can ensure~\ref{item:intervalstotheleft} by choosing $\rho$ sufficiently large follows from Proposition~\ref{prop:Fgrowth}. 
   Now let $\eps$ be the distance between the quadruples $Q$ and $\tilde{Q}$; i.e. 
    \[ \eps = \max(A - \tilde{A}, \tilde{B} - B , C - \tilde{C}, \tilde{D} - D). \]    
    Take $\tau > \tilde{D}$, and choose $\rho$ according to Proposition~\ref{prop:caratheodory}. We leave it to 
     the reader to verify that the claims follow from this choice. 
\end{proof}

\section{Proof of the main theorem}\label{sec:pseudoarcs}

  We are now ready to show the following result, which implies Theorem~\ref{thm:pseudoarc1}.
  
  \begin{thm}[Model function with pseudo-arcs]\label{thm:Fpseudoarcs}
    There is $F\in\mathcal{F}$ such that every Julia continuum of the periodic extension $\hat{F}$ is a pseudo-arc. 
    
    Moreover, $F$ can be chosen such that 
     \begin{equation}\label{eqn:Flowerorder} \liminf_{r\to\infty} \max_{\re \zeta = r} \frac{\log \re F(\zeta) }{r} = 1/2. \end{equation}
  \end{thm}

  We prove this theorem by inductively applying the following proposition. 
 \begin{prop}[Creating crookedness over prescribed quadruples]\label{prop:wiggles}
   Let $F\in\mathcal{F}_N$, let $Q$ be a quadruple, and let $\rho > 0$.

  Then there is $\tilde{F}\in \mathcal{F}_{\hat{N}}$ for some $\hat{N} > N$,  which is $(N,\rho)$-close to $F$, and such that
    $\U_n(Q,\tilde{F}) = \hat{\U}_n(Q,\tilde{F})$ for some $n\geq 0$. 
 \end{prop}
 
 \begin{proof}[Proof of Theorem~\ref{thm:Fpseudoarcs}, using Proposition~\ref{prop:wiggles}]
   Let $(Q_k)_{k=1}^{\infty}$ be the countably many quadrilaterals appearing in the hypothesis of 
    Theorem~\ref{thm:pseudoarcrealisation}. 
    We inductively construct a sequence $(F_k)_{k=0}^{\infty}$ with $F_k\in\mathcal{F}_{N_k}$ for an increasing sequence $(N_k)$; here $N_0 = 0$ and $F_0$ is the
       unique function in $\mathcal{F}_0$; i.e. the conformal map $F_0\colon S\to\HH$ with $F_0(5)=5$ and $F_0(\infty)=\infty$. 
       
  Along with the functions $F_k$, we construct a sequence $\rho_k$, in such a way that 
   \begin{enumerate}[(a)]
     \item $F_{k+1}$ is $(N_k,\rho_k)$-close to $F_k$;
     \item for every $k\geq 1$, there is some $n_k$ with the following property:
          if $F\in\mathcal{F}$ is $(N_k,\rho_k)$-close to $F_k$, then
          $\U_{n_k}(Q_{k} , F) = \hat{\U}_{n_k}(Q_{k},F)$.\label{item:crookedness}
   \end{enumerate}       
  The number $\rho_0$ can be chosen in an arbitrary manner. 

   Suppose that $F_k$ and $\rho_k$ have been constructed. Apply Proposition~\ref{prop:wiggles} to the function
      $F_k$, the value $\rho_k$, and a slightly smaller quadruple $\tilde{Q}_{k+1}<Q_{k+1}$, to obtain the function $F_{k+1}\in \mathcal{F}_{N_{k+1}}$.
      This function has the property that $\U_{n_{k+1}}(\tilde{Q}_{k+1}, F_{k+1})=\tilde{U}_{n_{k+1}}(\tilde{Q}_{k+1},F_{k+1})$ for some $n_{k+1}\geq 0$. 
      By Proposition~\ref{prop:Qcaratheodory}, there is $\rho_{k+1}$ such that 
      $\U_{n_{k+1}}(Q_{k+1},F) = \tilde{U}_{n_{k+1}}(Q_{k+1},F)$ for every $F$ that is $(N_{k+1},\rho_{k+1})$-close to $F_{k+1}$. 
       This completes the inductive construction. 
      
   Now let $F\in\mathcal{F}$ be the limit of the functions $F_k$; that is, $F$ has $N(F)=\infty$ and 
      is defined by the sequence $(r_k(F_{k+1}),R_k(F_{k+1}))_{k=0}^{\infty}$. Then
      $F$ is $(N_k,\rho_k)$-close to $F_k$ for every $k\geq 1$. By
      Lemma~\ref{lem:Unproperties}~\ref{item:allcrooked} and
      property~\ref{item:crookedness} of the inductive construction,
      we have $\U_{n}(Q_k,F) = \hat{\U}_n(Q_k,F)$ for all $n\geq n_k$.
        Hence $F$ satisfies the hypotheses of Theorem~\ref{thm:pseudoarcrealisation}, and 
      every Julia continuum of $\hat{F}$ is a pseudo-arc. 
      
    We claim that, if each $\rho_{k}$ is chosen sufficiently large, depending on 
     $R_{N_k-1} \defeq R_{N_k-1}(F) = R_{N_k-1}(F_k)$, then 
      $F$ additionally satisfies~\eqref{eqn:Flowerorder}. Indeed, let $a\in T(F)$ be such that
      $\re a = R_{N_k-1} + 1$ and such that $\lvert F(a)\rvert$ is maximal. Also let 
      $\zeta \in T(F)$ with 
         $\re \zeta = \rho_k - \nu_0$; we may assume that $\rho_k$
         is chosen so large that $\re \zeta > \re a + 2\nu_0$, and in particular $\lvert F(\zeta)\rvert > \lvert F(a)\rvert$.
          
          Let $\gamma_a$ and $\gamma_{\zeta}$ be the ``vertical'' geodesics through $a$ and $\zeta$, respectively,
          and let $Q\subset T$ be the quadrilateral bounded by $\gamma_a$ on the left, $\gamma_z$ on the right,
          and pieces of the upper and lower boundaries of $T$. Then $\log F$ maps $Q$ conformally to the rectangle
             \[ R \defeq \{ x + iy \colon \log \lvert F(a)\rvert < x < \log \lvert F(\zeta)\rvert \text{ and }
                  \lvert y \rvert < \pi/2 \}. \] 
         The conformal modulus of this rectangle is 
            \[ \mod(R) = \frac{1}{\pi}\cdot \log \frac{\lvert F(\zeta)\rvert}{\lvert F(a)\rvert}. \] 
         On the other hand,
            \[ Q \subset \{ x + iy\colon 0 < x < \re \zeta + \nu_0 \text{ and }\lvert y\rvert < \pi \}, \] 
         and hence $\mod(Q) \leq (\re \zeta + \nu_0)/(2\pi)$. Hence 
            \[ \log \lvert F(\zeta)\rvert \leq \frac{\re \zeta}{2} + \frac{\nu_0}{2} + \log \lvert F(a)\rvert \leq
                          \frac{\re \zeta}{2} + \frac{\nu_0}{2} + C\cdot (R_{N_k-1}+1), \]
       where $C$ is the constant from Proposition~\ref{prop:Fgrowth}. If $\rho_k$ is chosen so large that 
            \[ \frac{\nu_0}{2} + C\cdot (R_{N_k-1}+1) \leq \frac{ \rho_k-\nu_0}{n} = \frac{\re \zeta}{n}, \]
       then 
            \[ \log \lvert F(\zeta)\rvert \leq \re \zeta\cdot \left(\frac{1}{2} + \frac{1}{n}\right), \]
    and~\eqref{eqn:Flowerorder} follows. 
 \end{proof}
 
To prove Proposition~\ref{prop:wiggles}, we consider a quadrilateral $Q$, and proceed to inductively 
   introduce wiggles over the intervals in $\tilde{U}_n(Q,\phi)$. More precisely, we show the following.
   
   \begin{prop}[Increasing the number of intervals mapped crookedly]\label{prop:subwiggles}
    Let $F\in\mathcal{F}_N$, let $Q<\tilde{Q}$ be quadruples, and let $\rho > 0$. Let $m$ be the integer such that 
      \[ \# \U_n(Q,\phi) - \# \hat{\U}_n(Q,\phi) = m\]
     for sufficiently large $n$.
     
     Suppose that $m>0$.
   Then there is $\tilde{F}\in \mathcal{F}_{N+1}$,  which is $(N,\rho)$-close to $F$, such that
      \[ \# \U_n(\tilde{Q},\tilde{\phi}) - \# \hat{\U}_n(\tilde{Q},\tilde{\phi}) \leq m-1 \]
     for all sufficiently large $n$. 
   \end{prop} 
   \begin{proof}
    Let $n_0$ be so large that all elements of $\U_{n_0}(Q,\phi)$ are to the right of the last wiggle of $F$. Thus
     $\# \U_n(Q,\phi) = \# \U_{n_0}(Q,\phi) $ for all $n\geq n_0$ by Proposition~\ref{prop:Un1}. In particular, a number $m$ as in 
     the statement of the proposition does indeed exist. By assumption, if $n_0$ is chosen sufficiently large, then 
      \[ \# \hat{\U}_n(Q,\phi) = \# \hat{\U}_{n_0}(Q,\phi) =  \# \U_{n_0}(Q,\phi) - m\] 
      for $n\geq n_0$.

  Choose $\rho_1 > \rho$ so large that Proposition~\ref{prop:Qcaratheodory} applies, 
   and let $\rho_2 > \rho_1$ be so large that Proposition~\ref{prop:caratheodory} applies with $\tau = \rho_1$ and $\eps$ the 
     distance
      between $Q$ and $\tilde{Q}$ (as in the proof of Proposition~\ref{prop:Qcaratheodory}).
%  Let $\eps$ be the distance between the quadruples $Q$ and $\tilde{Q}$; i.e. 
%    \[ \eps = \max(A - \tilde{A}, B - \tilde{B}, C - \tilde{C}, D - \tilde{D}). \] 
%    Now choose $\tau$ and $\rho> \max(\tau,\rho_N)$ so large that
%      \begin{enumerate}[(a)]
%        \item All elements of $\U_{n_0}(Q,\phi)$ are to the left of $\tau$;
%        \item The conclusion of Proposition~\ref{prop:caratheodory} holds for $\tau$, $\rho$ and $\eps/2$;
%        \item The conclusions of Proposition~\ref{prop:Qcaratheodory} hold for $Q$, $\tilde{Q}$, $n_0$ and $\rho$. 
%      \end{enumerate}     
    Let $n_1$ be sufficiently large that all elements of $\U_{n_1}(Q,\phi)$ are to the right of $\rho_2 + \nu_0 + 1 $. Let 
      $[\hat{A},\hat{D}]$ be the right-most element of this set that is not in $\hat{\U}_{n_1}(Q,\phi)$, and choose two points 
       $\hat{B} < \hat{C}$ that map to the elements $B$ and $C$ of the quadruple $Q$ (not necessarily in that order). By 
       Lemma~\ref{lem:phiproperties}, we have 
       \[ \min( \hat{B} - \hat{A},\hat{D} - \hat{C}) \geq 2^{n_1}\cdot \lvert Q\rvert > 2 \nu_0 + 2 \]
       if $n_1$ was chosen sufficiently large. 
       Let $\hat{Q} = (\hat{A} < \hat{B} < \hat{C}<\hat{D})$ be the resulting 
       quadruple. 
       
     Define $\tilde{F}\in \mathcal{F}_{n+1}$ which is $(N,\rho_2)$-close to $F$ by setting 
      $r_N(\tilde{F}) \defeq \tilde{B} - \nu_0 $ and $R_N(\tilde{F}) \defeq \tilde{C} + \nu_0$. 
      
      \begin{claim}
        The following hold. 
    \begin{enumerate}[(1)]
       \item $\# \U_n(\tilde{Q},\tilde{\phi}) \leq \U_{n_0}(Q,\phi)$ for $n\geq n_0$.
       \item Every element of $\U_{n_0}(\tilde{Q},\tilde{\phi})$ that contains an element of $\hat{\U}_{n_0}(Q,\phi)$ is in $\hat{\U}_{n_0}(\tilde{Q},\tilde{\phi})$.\label{item:crookedinherited}
       \item We have $\U_1( \hat{Q},\tilde{\phi}) = \hat{\U}_1(\hat{Q},\tilde{\phi})$.
    \end{enumerate} 
      \end{claim}
      \begin{subproof}
        We have $\# \U_{n_0}(\tilde{Q},\tilde{\phi}) \leq \# \U_{n_0}(Q,\phi)$ by Proposition~\ref{prop:Qcaratheodory}, and 
          every interval of the former contains an interval of the latter. Part~\ref{item:crookedinherited} also follows from
          Proposition~\ref{prop:Qcaratheodory}. 
          
        Moroever, for $n\geq n_0$, 
          no interval of $\U_{n}(Q,\phi)$ is contained in $[r_N(\tilde{F})-\nu_0, R_N(\tilde{F})+\nu_0]$. Indeed,
          otherwise this interval would be contained in $[\hat{A},\hat{D}]\in \U_{n_1}(Q,\phi)$, contradicting Lemma~\ref{lem:Unproperties}~\ref{item:notnested}. 
          By Proposition~\ref{prop:Qcaratheodory}, also no interval of $\U_n(\tilde{Q},\tilde{\phi})$ is contained in $[r_N(\tilde{F})-\nu_0, R_N(\tilde{F})+\nu_0]$.
          So by Proposition~\ref{prop:Un1}, we have 
          \[ \# \U_n(\tilde{Q},\tilde{\phi}) \leq \# \U_{n_0}(\tilde{Q},\tilde{\phi}) \leq \U_{n_0}(Q,\phi) \] 
          for $n\geq n_0$. 

      Finally, let  $J\in \U_1(\hat{Q},\tilde{\phi})$. Then $F^{-1}(J)$ is an arc that connects a point at real part $\hat{A}< r_N(\tilde{F})$ to a point at real part 
          $\hat{D}> R_N(\tilde{F})$. By the shape of the tract (see Figure~\ref{fig:tracts}), along this curve the real parts must first grow to at least 
            $R_N(\tilde{F})-1$, decrease again at least to $r_N(\tilde{F})+1$, before they can reach $\hat{D}$. It follows that 
            $\tilde{\phi}$ does indeed map $J$ crookedly to $\hat{Q}$. 
      \end{subproof} 
     Now we can conclude the proof. Let $n > n_1$, and set $k\defeq n - n_0$. 
        By Proposition~\ref{prop:Qcaratheodory}, for every $J\in \U_n(\tilde{Q},\tilde{\phi})$, there
       is $I\in\U_{n_0}(Q,\phi)$ such that $I\subset \tilde{\phi}^{k}(J)$. Moreover, by the claim,
      there is at most one such $J$
        for every $I$.

      \begin{enumerate}[(a)]
        \item If $I\in \hat{\U}_{n_0}(Q,\phi)$, then by the claim,
           $\tilde{\phi}^k(J)\in \hat{\U}_{n_0}(\tilde{Q},\tilde{\phi})$, and thus
           $J \in \hat{\U}_{n}(\tilde{Q},\tilde{\phi})$. 
        \item If $I = \phi^{n_1-n_0}([\hat{A},\hat{D}])$, then $\tilde{\phi}^{n-n_1} \supset [\hat{A},\hat{D}]$. By the claim, 
            $\tilde{\phi}^{n-n_1}$ maps $I$ crookedly over $\hat{Q}$. By choice of $\rho$, and Proposition~\ref{prop:caratheodory}, one of
             $\tilde{\phi}^{n_1}(\hat{B})$ and $\tilde{\phi}^{n_1}(\hat{C})$ is between $\tilde{A}$ and $\tilde{B}$, and the other between
             $\tilde{C}$ and $\tilde{D}$. Thus $\phi^n$ maps $I$ crookedly over $\tilde{Q}$. 
      \end{enumerate}            
      In conclusion, at least one more of the intervals in $\U_{n}(\tilde{Q},\tilde{\phi})$ is mapped crookedly than was the case
         for $\U_n(Q,\phi)$, and the proof is complete. 
   \end{proof} 

\begin{proof}[Proof of Proposition~\ref{prop:wiggles}]
  Let $S_0$ be a quadruple with $S_0<Q$. Let $m$ be the integer such that 
    \[ \# \U_n(S_0,\phi) - \# \hat{\U}_n(S_0,\phi) = m \]
   for all sufficiently large $n$, as in Proposition~\ref{prop:subwiggles}. 
      
  Let $S_0 < S_1 < \dots < S_m = Q$ be a sequence of quadruples. Inductively 
    apply Proposition~\ref{prop:subwiggles} to obtain a sequence of functions 
    $F_0 = F, F_1,\dots, F_k \eqdef \tilde{F}$, with $k\leq m$ and $F_j \in \tilde{F}_{N+j}$, such that 
      \[ m_j \defeq \lim_{n\to\infty} \# \U_n(S_j, \phi_j) - \# \hat{\U}_n(S_j,\phi_j) \] 
     is strictly decreasing, and $m_k = 0$. Then $\tilde{F}$ is the desired function. 
\end{proof} 

\begin{proof}[Proof of Theorem~\ref{thm:pseudoarc1}]      
 Let $F$ be as in Theorem~\ref{thm:Fpseudoarcs}, and $\hat{F}$ its $2\pi i$-periodic extension. 
      By Theorem~\ref{thm:realisation}, there is a disjoint-type entire function for which every Julia continuum is homeomorphic to a Julia continuum of $\hat{F}$, 
       and hence also a pseudo-arc. 
\end{proof}

\begin{proof}[Proof of Theorem~\ref{thm:pseudoarc2}]
 The theorem also follows from the construction in the next section, which uses Bishop's \emph{quasiconformal folding}. 
   Here, we indicate how to prove the theorem without recourse to quasiconformal folding.    
   Firstly, note that it follows from the proof of Theorem~\ref{thm:realisation} that the entire function $g$ 
     constructed has finite lower order. Indeed, the function $g$ is
    \emph{quasiconformally equivalent near infinity} to the universal covering $f$ defined by 
       \[ f(\exp(\zeta)) = \exp(F(\zeta)), \]
    and the latter has finite lower order. (Compare~\cite[Proof of Proposition 2.3(b)]{area}.) 
       
    However, with this argument we cannot directly show that the lower order can be to be $1/2$. To do so, we modify the proof by changing the functions in the class $\mathcal{F}$. Instead of
     considering conformal isomorphisms $T\to \HH$, where $T$ is a tract as in Figure~\ref{fig:tracts}, we instead consider conformal isomorphisms
       \[ F\colon T\to H \defeq \{ x + iy \colon x > -14 \log_+\lvert y\rvert\}, \]
       again normalised so that $F(5)=5$ and $F(\infty)=\infty$. This does not affect the proof of the theorem. Indeed, we only used 
       the range of the function $F$ in certain 
       estimates of the hyperbolic metric, which hold equally for $F$ with range $H$ as above. 

      Then we apply Theorems 1.7, 1.8 and 1.9 of~\cite{approximationhypdim} to the function $F$, to obtain an entire function whose Julia continua are all pseudo-arcs,
       and which has order $1/2$ by~\eqref{eqn:Flowerorder}. 
\end{proof}

\section{Finitely many singular values} \label{sec:speiser}

  In this section, we indicate how to modify the construction to prove Theorem~\ref{thm:speiser}.
    We follow similar lines as \cite[Section~18]{bishopfolding} and \cite[Section~15]{arclike}, and 
    assume that the reader is familiar with Bishop's technique of quasiconformal folding from~\cite{bishopfolding}.

   \begin{proof}[Proof of Theorem~\ref{thm:speiser}]
      Let $F\colon T\to \HH$ be the function $F\in\F$ from Theorem~\ref{thm:Fpseudoarcs}. We modify 
        $F$ to a function $\tilde{F}\colon \tilde{T}\to\HH$, where 
          \[ \tilde{T} = T \cup \{x + iy\colon x\leq 4 \text{ and } \lvert y\rvert < \pi\}, \]
       and $\tilde{F}$ is chosen such that $F(5)=5$, and $F(z)\to\infty$ as $\re z\to\infty$ in $\tilde{T}$. 
       
       We do not have $\tilde{T}\subset \overline{\HH}$, so $\tilde{F}$ is not a function in $\mathcal{H}$. 
         However,~\eqref{eqn:stripestimate}, and hence 
          Lemma~\ref{lem:expansion}, also holds for $\tilde{F}$. It follows that Lemma~\ref{lem:phiproperties} also
          holds for the one-dimensional projection $\tilde{\phi}$ of $\tilde{F}$. Moreover, the maps $\phi$ and
          $\tilde{\phi}$ behave similarly:
          
      \begin{claim}
        There exists constants $C,X>0$ with the following property. If $x\geq 9$ and $n\geq 1$ 
           with $\tilde{\phi}^n(x) \geq X$, then
          \[ \lvert \tilde{\phi}^n(x) - \phi^n(x)\rvert \leq C. \]
      \end{claim}
      \begin{subproof}
         Let $\phi \colon \tilde{T} \to T$ be a quasiconformal homeomorphism
          such that $\phi(z)=z$ when $\re z\geq 5$; such a map is
         easy to construct. Then we have 
            \[ \psi \circ \tilde{F} = F\circ \phi, \]
        where $\psi$ is a quasiconformal self-map of the upper half-plane. The dilatation of $\psi$ is supported on a bounded
        set, and thus $\psi(z) \approx c\cdot z$ as $z\to\infty$, for some $c>0$. 
        
        Hence the hyperbolic distance between $x$ and $\psi(x)$ is uniformly bounded (for $x\geq 5$, say), 
         and therefore the Euclidean distance between $F^{-1}(x)$ and $F^{-1}(\psi(x))$ is uniformly bounded by
         a constant $\delta$. Recall that $F^{-1}(\psi(x)) = \tilde{F}^{-1}(x)$, as long as this point as real part at least $5$.
         
         Set $X\defeq 9 + \delta$ and $C \defeq 2\delta$. 
         Recall that $\tilde{\phi}^n(x)\geq 9+\delta $ implies $\tilde{\phi}^k(x)\geq 9+\delta$ for $k=0,\dots n$, by
          Lemma~\ref{lem:phiproperties}. 
          
       The claim now follows by induction. If it holds for $n$, then we have 
         \begin{align*} \lvert \tilde{\phi}^{n+1}(x) - \phi^{n+1}(x)\rvert &\leq 
             \lvert \tilde{\phi}^{n+1}(x) - \phi(\tilde{\phi}^n(x))\rvert + \lvert \phi(\tilde{\phi}^n(x)) - \phi^{n+1}(x)\rvert  \\
               &\leq \lvert \tilde{F}^{-1}(\tilde{\phi}^n(x)) - F^{-1}(\tilde{\phi}^n(x))\rvert + \delta 
                  \leq 2\delta = C, \end{align*}
       as desired. 
      \end{subproof}
      
      Now consider the half-plane $H = \{a + ib\colon a > 4\}$ and the $2\pi i$-periodic extension $G$ of the restriction 
         \[ \tilde{F} \colon \tilde{F}^{-1}(H) \to H. \]
       We have $\overline{G^{-1}(H)} \subset H$, at least if $r_1(F)$ was chosen sufficiently large. Indeed, on the initial part
        of the tract, the function $\tilde{F}$ is close to the conformal isomorphism
          \[ F_S \colon \tilde{S} = \{ a + ib\colon \lvert b\rvert < \pi\} \to \HH; \qquad z\mapsto 5\exp\left(\frac{z-5}{2}\right), \]
          and for $z\in H$ we have
             \[ \re F_S^{-1}(z) = 5 + 2\cdot \log \frac{\lvert z\rvert}{5} \geq 5 + 2\log \frac{4}{5} = 5 - \log \frac{25}{16} > 4. \]
             
      So the map $G$ is a disjoint-type function in the class $\Blog^{\operatorname{p}}$ as defined in~\cite[Definition~3.3]{arclike}, and it makes
       sense to consider its Julia continua, exactly as we did for the function $F$. That is, a Julia continuum of $G$ is 
       a connected component of the set of points that remain in $H$ under iteration of $\tilde{F}$. 
       
      It follows from the claim above that $\tilde{\phi}$ satisfies the hypotheses of Theorem~\ref{thm:pseudoarcrealisation}, 
       and the proof
       of that theorem applies equally to $G$. So all of the Julia continua of $G$ are pseudo-arcs. 
       
     Now we apply quasiconformal folding to the function $\tilde{F}$. To do so, 
       introduce vertices on $\tilde{T}$ so that $G = \exp(\tilde{T})$ 
        is a bounded-geometry tree in the sense of~\cite{bishopfolding}. 
       These vertices should be placed such that the two intervals on $\partial \HH$ corresponding to an edge have length
       at least $\pi$ under $\tilde{F}$. We may use standard estimates of harmonic measure to deduce that 
       this can be achieved while letting the length of any edge at real parts $\geq R$ tend to zero like $\exp(c\cdot R)$,
       for some $c>0$. 
       
    The quasiconformal folding theorem~\cite[Theorem~1.1]{bishopfolding} now yields an entire function
     $f$, having at most two critical values and no asymptotic values, and a $K$-quasiconformal map $\phi$ such that 
       \[ f \circ \phi = \cosh \circ \tilde{F} \] 
         on the complement of the set 
           \[ G(r_0) \defeq \bigcup_{e} \{z\in\C\colon \dist(z,e) < r_0 \diam(e) \}, \]
           where the union is over all edges of $G$, and $r_0$ is a universal constant. In particular, the dilatation of the 
           map $\phi$ is supported on the set $G(r_0)$. Precomposing $\phi$ with a 
         suitable linear map, we may choose $f$ to have disjoint type. 
           
       Our assumption on the lengths of edges on $\partial \tilde{T}$ implies that the cylindrical area of
        $G(r_0)$ is finite, and therefore by the Teichm\"uller-Wittich-Belinski-Lehto Theorem, the map $\phi$ is
         asymptotically conformal at infinity. It follows that the map $f$ has lower order $1/2$. 
         
      It remains to show that the Julia continua of $f$ are all pseudo-arcs. Suppose, by contradiction, that $K$ was a decomposable
        subcontinuum of $J(f)$. Then, as in the proof of Theorem~\ref{thm:realisation}, we can find a subcontinuum
        $\tilde{K}\subset K$, for some $n\geq 0$, which is also decomposable, and which furthermore has the following
        properties: 
        \begin{enumerate}[(a)]
          \item the iterates of $f$ tend to infinity uniformly on $\tilde{K}$; 
          \item for sufficiently large $n$, the sets $f^n(\tilde{K})$ are disjoint from $G(r_0)$. 
        \end{enumerate}
        
        As noted in \cite[Section~15]{arclike}, 
          any such subcontinuum of $J(f)$ is homeomorphic to a subcontinuum of $J(G)$. Therefore $\tilde{K}$ is
          a pseudo-arc, which is a contradiction to the assumption that $\tilde{K}$ is decomposable. 
   \end{proof} 
   
 \begin{proof}[Proof of Corollary~\ref{cor:locallyconnected}]
   Let $g$ be the function from Theorem~\ref{thm:speiser}; we may assume that its critical values are $0$ and $1$. 
     Since $g$ has no asymptotic values, the 
     set $T\defeq g^{-1}([0,1])$ has the structure of an infinite tree, with vertices at the preimages of $0$ and $1$. 
       
   The function $g$ has infinitely many critical points over both $0$ and $1$. This  follows immediately from the quasiconformal folding 
      technique, but we may also see this directly: Every vertex of degree $1$ is a leaf of the tree $T$; since 
      $T$ is connected, it follows that every element of $g^{-1}(0)$ is adjacent to at least one element of $g^{-1}(1)$ of
      degree $\geq 2$, and vice versa.

    In particular, we may pick two critical points $c_0$ and $c_1$ with $g(c_0)=0$ and $g(c_1)=1$ such that 
       $c_0$ and $c_1$ are not adjacent in $T$. Define
         \[ f(z) \defeq g\left(z\cdot c_1 + (1-z)\cdot c_0\right). \]
    Then $0$ and $1$ are super-attracting fixed points of $f$. Since these are the only
     critical values of $f$, it follows that $f$ is \emph{hyperbolic} in the sense of~\cite[Definition~1.1]{hyperbolicboundedfatou}. 
     By~\cite[Corollary~1.9]{hyperbolicboundedfatou}, the Julia set $J(f)$ is locally connected. Moreover,
      every component of $F(f) =\C\setminus J(f)$ is a Jordan curve~\cite[Theorem~1.4]{hyperbolicboundedfatou}
       
     To show that $J(f)$ is a Sierpi\'nski carpet, we must show that different components of 
        $F(f) = \C\setminus J(f)$ have pairwise disjoint closures. Let $U_0$ and $U_1$ be the components
        containing $0$ and $1$, respectively. Recall that $U_0$ is a Jordan domain, and that $\overline{U_0}$ contains
        no critical points other than $0$. So we can find a Jordan domain $\tilde{U}_0\supset \overline{U_0}$ whose closure 
        also contains no other critical points. Every connected component of 
        $f^{-1}(\tilde{U}_0)$ is mapped to $\tilde{U}_0$ as a finite-degree proper map with a single critical point, and 
        in particular contains a unique connected component of $f^{-1}(U_0)$. Therefore different connected components
        of $f^{-1}(U_0)$ have disjoint closures. The same argument works for components of $f^{-1}(U_1)$. 

    Next we show that $\partial U_0$ and $\partial U_1$ are disjoint. Suppose, by contradiction, that
      there was a common point $z_0$. There can be at most one such point. Indeed, otherwise
      there would be a Jordan curve in $\overline{U_1}\cup \overline{U_0}$ that is contained in 
        $K\defeq \overline{U_1}\cup \overline{\U_2}$. Since $K$ is bounded and forward-invariant,
        the interior region of this curve would be contained in the Fatou set, which contradicts the fact that $0$ and $1$ belong
        to different Fatou components. 
        
     We may find invariant curves $\gamma_0$ and $\gamma_1$ such that $\gamma_j$ connects $z_0$ to $j$ in $U_j$. 
       By considering a homotopy between $\gamma_0\cup \gamma_1$ and $[0,1]$, we see that $0$ and $1$ are
       adjacent in the graph $f^{-1}([0,1])$. This contradicts the choice of $c_0$ and $c_1$, and the definition of $f$. 
       So $\partial U_0$ and $\partial U_1$ are disjoint. This implies that different 
       connected components of the backwards orbits of $U_0$ and $U_1$ must also be disjoint.

   The existence of invariant pseudo-arcs in $J(f)$ follows from \cite{boettcher}. Indeed, 
      the main result of that paper shows the existence of a homeomorphism that conjugates the
        disjoint-type map $g$ to $f$ on the set $J_{\geq R}(g)$ 
        of points whose orbit under $g$ remains outside a certain disc $D(0,R)$ 
        at all times. This set contains infinitely many different invariant Julia continua of $g$, which are all
        pseudo-arcs by our main theorem. 
        
    Similarly, the collection of all Julia continua of $g$ that are completely contained in $J_{\geq R}(g)$ 
       give rise to the uncountable collection of pseudo-arcs whose existence is asserted in the final part of the corollary. 
 \end{proof}
   
\bibliographystyle{amsalpha}
\bibliography{../../Biblio/biblio}

\end{document}